\documentclass[10pt]{article}
\usepackage[hmargin=3.5cm,vmargin=3cm]{geometry}
\usepackage{amsthm}
\usepackage{amsmath}
\usepackage{amssymb}
\usepackage{enumitem}
\usepackage{titlesec}
\usepackage[pdftex]{color}
\definecolor{darkblue}{rgb}{0.3,0.3,0.7}
\usepackage[pdftex,final]{graphicx}
\usepackage[pdftex,pdftitle={Limit theorems for power variations of ambit fields driven by white noise},pdfauthor={Mikko S. Pakkanen},colorlinks=TRUE,allcolors=darkblue]{hyperref}

\titleformat{\section}[block]{\center\bfseries}{\thetitle.}{10pt}{}{}

\titleformat{\subsection}[hang]{\bfseries}{\thetitle.}{10pt}{}{}

\titleformat{\subsubsection}[hang]{\it}{\thetitle.}{10pt}{}{}

\numberwithin{equation}{section}

\theoremstyle{plain}
\newtheorem{thm}[equation]{Theorem}

\newtheorem{lem}[equation]{Lemma}

\newtheorem{prop}[equation]{Proposition}
\theoremstyle{remark}
\newtheorem{rem}[equation]{Remark}
\theoremstyle{definition}
\newtheorem{exm}[equation]{Example}
\newtheorem{asm}[equation]{Assumption}
\newtheorem{defn}[equation]{Definition}

\newcommand{\ud}{\mathrm{d}}
\newcommand{\supp}{\mathrm{supp}}

\newcommand{\eqdefl}{\mathrel{\mathop:}=}
\newcommand{\N}{\mathbb{N}}

\newcommand{\R}{\mathbb{R}}
\newcommand{\Z}{\mathbb{Z}}

\newcommand{\E}{\mathbf{E}}
\newcommand{\Cov}{\mathbf{Cov}}
\newcommand{\Var}{\mathbf{Var}}
\newcommand{\Corr}{\mathbf{Corr}}
\newcommand{\prob}{\mathbf{P}}
\newcommand{\Hil}{\mathcal{H}}
\newcommand{\wn}{W}
\providecommand{\bs}[1]{\boldsymbol{#1}}
\newcommand{\revised}[1]{\color{magenta}#1\color{black}}
\renewcommand{\revised}[1]{#1}

\newcommand{\myfigure}[2]{\center \includegraphics*[#1]{#2} }

\begin{document}
\title{\bf Limit theorems for power variations \\ of ambit fields driven by white noise}
\author{
Mikko S. Pakkanen\thanks{CREATES and Department of Economics and Business, 
Aarhus University, 
Fuglesangs All\'e 4, 
8210 Aarhus V, Denmark, 
URL: \url{http://www.mikkopakkanen.fi/},
\revised{E-mail: \href{mailto:mpakkanen@econ.au.dk}{\nolinkurl{mpakkanen@econ.au.dk}}.}}
}

\date{January 7, 2014}

\maketitle

\begin{abstract}
We study the asymptotics of lattice power variations of two-parameter ambit fields driven by white noise. Our first result is a law of large numbers for power variations. Under a constraint on the memory of the ambit field, normalized power variations converge to certain integral functionals of the volatility field associated to the ambit field, when the lattice spacing tends to zero. This result holds also for thinned power variations that are computed by only including increments that are separated by gaps with a particular asymptotic behavior. Our second result is a stable central limit theorem for thinned power variations.
\end{abstract}


\vspace*{1ex}


\section{Introduction}

\subsection{Ambit fields and volatility}

A characteristic feature of many real-world random phenomena is that the \emph{magnitude} or the \emph{intensity} of realized fluctuations varies in time or space, or both. There are various terms used in different contexts that roughly correspond to this characteristic. To highlight two of them, in studies of turbulence, this is called \emph{intermittency}, whereas in finance and economics the corresponding notion is (stochastic) \emph{volatility}. Sudden extreme fluctuations --- say, rapid changes in wind velocity or prices of financial securities --- have often dire consequences, so understanding their statistical properties is clearly of key importance.

Barndorff-Nielsen and Schmiegel \cite{BS2005,BS2006} have introduced a class of L\'evy-based random fields, for which they coined the name \emph{ambit field}, to model space-time random phenomena that exhibit intermittency or stochastic volatility.
The primary application of ambit fields has been phenomenological modeling of turbulent velocity fields. Additionally, Barndorff-Nielsen, Benth, and Veraart \cite{BNBV2010} have recently applied ambit fields to modeling of the term structure of forward prices of electricity. Electricity prices, in particular, are prone to rapid changes and spikes since the supply of electricity is inherently inelastic and electricity cannot be stored efficiently.
It is also worth mentioning that, at a more theoretical level, some ambit fields have been found to arise as solutions to certain stochastic partial differential equations \cite{BNBV2011}. Barndorff-Nielsen, Benth, and Veraart \cite{BBV2012} provide a survey on recent results on ambit fields and related \emph{ambit processes}.

In this paper, we study the asymptotic behavior of power variations of a two-parameter ambit field driven by white noise, with a view towards measuring the realized volatility of the ambit field. 
Specifically, we consider ambit field $(Y_{(s,t)})_{(s,t)\in [0,1]^2}$, defined via the equation 
\begin{equation}\label{eq:genambitfield}
Y_{(s,t)} = \int_{A(s,t)} g(s-u,t-v) \sigma_{(u,v)} W(\ud u, \ud v),
\end{equation}
where the integrator $W$ is a white noise on $\R^2$ and the integrand is defined in terms of a positive-valued, continuous \emph{volatility field} $(\sigma_{(s,t)})_{(s,t) \in \R^2}$ and a \emph{weight function} $g \in L^2(\R^2)$. The integral in \eqref{eq:genambitfield} is computed over the set $A(s,t)\subset \R^2$, which is known as the \emph{ambit set} associated to the point $(s,t)$. More figuratively, $A(s,t)$ defines the ``ambit'' of noise and volatility innovations that influence $Y_{(s,t)}$. We use here the common specification of $A(s,t)$ as a translation of some fixed Borel set $A \subset \R^2$, viz.,
\begin{equation}\label{eq:ambitset}
A(s,t) \eqdefl A + (s,t) \eqdefl \{ (u+s,v+t) : (u,v) \in A \}.
\end{equation}
The shape of the set $A$ has a strong influence on the probabilistic properties of $Y$. When the parameter $t$ is interpreted as time, it is customary to assume that $A \subset \R \times (-\infty,0]$, so that only past innovations can influence the present. We refer to \cite{BNBV2011} for a discussion on the possible shapes of $A$ in various modeling contexts. 
We consider here only the case where the volatility field $\sigma$ and the white noise $W$ are \emph{independent}. In this case the integral in \eqref{eq:genambitfield} can be defined in a straightforward manner as a \emph{Wiener integral}, conditional on $\sigma$. (Ambit fields with volatilities that do depend on the driving white noise can be defined, but then the integration theory becomes more involved, see \cite{BNBV2011} for details. Moreover, the general framework of ambit fields also accommodates non-Gaussian random measures, L\'evy bases, as driving noise.)

The power variations we study are defined over observations of $Y$ on a \emph{square lattice} in $[0,1]^2$ using \emph{rectangular increments} (see Section \ref{subs:pv} for precise definitions). 
The spacing of the square lattice is $1/n$, and we let $n \rightarrow \infty$ in the asymptotic results.
In addition to ordinary power variations that involve all of the available increments, we consider also \emph{thinned} power variations that are computed using  only every $k_n$-th increment in the lattice. Asymptotically, we let $k_n \rightarrow \infty$ so that $k_n/n \rightarrow 0$. Similar procedures have been considered in the context of Gaussian processes by Lang and Roueff \cite{LR2001} and, more recently, in the context of Brownian semistationary processes by Corcuera et al.\ \cite{CHPP2012}.

Our first result is a functional law of large numbers for both ordinary and thinned power variations (Theorem \ref{thm:powerlln}). Under an assumption that constrains the memory of $Y$ through the so-called \emph{concentration measures} associated to the weight function $g$ (Assumption \ref{asm:lln}), we show that the suitably scaled power variation of $Y$ converges in probability to an integral functional of the volatility field $\sigma$.  Under a more restrictive and quantitative version of Assumption \ref{asm:lln} (which appears as Assumption \ref{asm:clt}), we also obtain a stable functional central limit theorem for thinned power variations (Theorem \ref{thm:powerclt}) with a conditionally Gaussian random field as the limit.
We give some explicit examples of weight functions $g$ that satisfy Assumptions \ref{asm:lln} or \ref{asm:clt} in Section \ref{subs:weight}.

The motivation of this paper is twofold. On the one hand, the study of the asymptotics of power variations of ambit fields is interesting from a probabilistic perspective, as it provides information on the fine structure of the realizations of ambit fields. On the other hand, in practical situations it is of interest to draw inference of volatility statistics of the form
\begin{equation}\label{eq:functional}
\int_0^{s} \int_0^t \sigma_{(u,v)}^p \ud u  \ud v,
\end{equation}
for $p>0$, based on discrete observations of the ambit field $Y$. Our law of large numbers establishes a sufficient condition that the suitably scaled $p$-th power variation of $Y$ over $[0,s]\times[0,t]$ converges to \eqref{eq:functional}. This could be seen as a first step towards a theory of volatility estimation for ambit fields.

\subsection{Related literature}

There is a wealth of literature on laws of large numbers and central limit theorems for power, bipower, and multipower variations of (one-parameter) stochastic processes. Notably, semimartingales are well catered for, see the monograph by Jacod and Protter \cite{JP2012} for a recent survey of the results. Similar results for non-semimartingales are, for obvious reasons, more case-specific. Closely relevant to the present paper are the results for Gaussian processes with \emph{stationary increments} \cite{BCP2009,BCPW2009} and \emph{Brownian semistationary processes} \cite{BCP2011,BCP2012,CHPP2012}. In fact, a Brownian semistationary process is the one-parameter counterpart of an ambit field driven by white noise. The proofs of the central limit theorems in \cite{BCP2009,BCP2011,BCP2012,BCPW2009,CHPP2012} use a method that involves Gaussian approximations of \emph{iterated Wiener integrals}, due to Nualart and Peccati \cite{NP2005}. We employ a similar approach, adapted to the two-parameter setting, in the proof of our central limit theorem.

Barndorff-Nielsen and Graversen \cite{BNG2011} have recently obtained a law of large numbers for the quadratic variation of an ambit process driven by white noise in a space-time setting. The probabilistic setup they consider is identical to ours, but their quadratic variation is defined over observations along a line in two-dimensional space-time, instead of a square lattice. The proof of our law of large numbers is inspired by the arguments used in \cite{BNG2011}.

Compared to the one-parameter case, asymptotic results for lattice power variations of random fields with two or more parameters are scarcer. 
There are, however, several results for Gaussian random fields, under various assumptions constraining their covariance structure.
Kawada \cite{Kaw1975} proves a law of large numbers for general variations of a class of multi-parameter Gaussian random fields, extending an earlier result of Berman \cite{Ber1967}. 
Guyon \cite{Guy1987} derives a law of large numbers for power variations (using two kinds of increments) of a stationary, two-parameter Gaussian random field with a covariance that behaves approximately like a power function near the origin.

An early functional central limit theorem for quadratic variations of a multi-parameter Gaussian random field, is due to Deo \cite{Deo1989}. 
Motivated by an application to statistical estimation of fractal dimension, Chan and Wood \cite{CW2000} prove a central limit theorem for quadratic variations of a stationary Gaussian random field satisfying a covariance condition that is somewhat similar to the one of Guyon \cite{Guy1987}.
More recently, R\'eveillac \cite{Rev2009b,Rev2009} has obtained central limit theorems for weighted quadratic variations of ordinary and fractional Brownian sheets. Similar results, which include also non-central limit theorems, applying to more general \emph{Hermite variations} of fractional Brownian sheets appear in the papers by Breton \cite{Bre2011} and R\'eveillac, Stauch, and Tudor \cite{RST2012}.

\section{Definitions and main results}\label{sec:mainresults}

\subsection{Notation}

For any $\bs{z}\in \R^2$, non-empty $A \subset \R^2$, and $r>0$, we write $B(\bs{z},r) \eqdefl \{ \bs{\zeta} \in \R^2 : \|\bs{\zeta}-\bs{z}\| < r \}$, and $A^r \eqdefl \bigcup_{\bs{\zeta}\in A} B(\bs{\zeta},r)$.
Moreover, $\overline{A}$ stands for the closure of $A$ in $\R^2$.

For any $s$,\ $t\in\R$, we use $s\wedge t \eqdefl \min(s,t)$ and $s\vee t \eqdefl \max(s,t)$, $\lfloor s \rfloor \eqdefl \max\{ r \in \Z : r \leqslant s \}$, $\lceil s \rceil \eqdefl \min\{ r \in \Z : r \geqslant s \}$, and $\{ x \} \eqdefl x - \lfloor x \rfloor$. It will be convenient to write $s \lesssim_\theta t$ (resp.~$s \gtrsim_\theta t$) whenever there exists $C_\theta >0$ that depends only on the parameter $\theta$, such that $s \leqslant C_\theta t$ (resp.~$C_\theta s \geqslant  t$).  We write $s \asymp_\theta t$ to signify that both $s \lesssim_\theta t$ and $s \gtrsim_\theta t$ hold.  

We denote the weak convergence of probability measures by $\stackrel{w}{\rightarrow}$, the convergence of random elements in law by $\stackrel{L}{\rightarrow}$, and the space of Borel probability measures on $\R^2$ by $\mathcal{P}(\R^2)$.
The support of $\nu \in \mathcal{P}(\R^2)$, or briefly $\supp \, \nu$, is the smallest closed set with full $\nu$-measure, given by $\bigcap_{r > 0}\big\{ \bs{z} \in \R^2 : \nu\big(B(\bs{z},r)\big)>0\big\}$. The Lebesgue measure on $\R^d$ is denoted by $\lambda_d$ and the Dirac measure at $\bs{z}\in\R^d$ by $\delta_{\bs{z}}$.

For any $q>0$, we write $m_q \eqdefl \E[|X|^q]$, where $X \sim N(0,1)$.
Finally, $|A|$ stands for the number elements in a finite set $A$, and we use the conventions $\N \eqdefl \{1,2,\ldots \}$ and $\N_0 \eqdefl \N \cup \{ 0\}$.

\subsection{Rigorous definition of the ambit field}

Let $\wn$ be a \emph{white noise} on $[-1,1]^2$ with the Lebesgue measure $\lambda_2$ as the \emph{control measure}. Recall that this means that $\wn$ is a zero-mean Gaussian process indexed by $\mathcal{B}([-1,1]^2)$ with covariance $\E[\wn(A)\wn(B)] = \lambda_2(A \cap B)$ for any $A$,\ $B\in \mathcal{B}([-1,1]^2)$. Throughout this paper, we consider an ambit field $Y$ given by
\begin{equation}\label{eq:ambitfield}
Y_{(s,t)} \eqdefl \int g(s-u,t-v) \sigma_{(u,v)} \wn(\ud u, \ud v), \quad (s,t) \in [0,1]^2,
\end{equation}
where $g \in L^2(\R^2)$ is a non-vanishing \emph{weight function} and $(\sigma_{(s,t)})_{(s,t)\in [-1,1]^2}$ is a continuous, strictly positive \emph{volatility field}, independent of $\wn$. Let us denote by $A_g$ the \emph{essential support} of $g$ (see,\ e.g.,\ \cite[p.\ 13]{LL2001} for the definition). In \eqref{eq:ambitfield} it suffices to integrate over the set
\begin{equation*}
A(s,t)\eqdefl -A_g+(s,t).
\end{equation*}
Thus, we recover the setting outlined in \eqref{eq:genambitfield} and \eqref{eq:ambitset} with $A = -A_g$. To ensure that $A(s,t) \subset [-1,1]^2$ for all $(s,t) \in [0,1]$, we assume that $A_g \subset [0,1]^2$.

The stochastic integral in \eqref{eq:ambitfield} is a conditional \emph{Wiener integral} with respect to $\wn$, defined as follows. Due to the independence of $\sigma$ and $\wn$, we may assume without loss of generality that the underlying probability space is the completion of the product space 
\begin{equation*}
(\Omega_\wn\times \Omega_\sigma, \mathcal{F}_\wn\otimes \mathcal{F}_\sigma,\prob_\wn \otimes \prob_\sigma),
\end{equation*}
where $(\Omega_\wn, \mathcal{F}_\wn,\prob_\wn)$ carries the white noise $\wn$, so that 
$\mathcal{F}_\wn = \sigma \{ W(A) : A \in \mathcal{B}([-1,1]^2)\}$,
and 
\begin{equation*}
(\Omega_\sigma, \mathcal{F}_\sigma,\prob_\sigma) = \big(C([-1,1]^2,\R_+), \mathcal{B}(C([-1,1]^2,\R_+)),\prob_\sigma\big)
\end{equation*}
 is the canonical probability space of $\sigma$, i.e., $\sigma_{(x,t)}(\omega) \eqdefl \omega(s,t)$ for any $\omega \in \Omega_\sigma$ and $(s,t) \in [-1,1]^2$. Then, for any $\omega \in \Omega_\sigma$ and $(s,t) \in [-1,1]^2$, we define $Y_{(s,t)}(\cdot,\omega)$ to be the Wiener integral of the function $(u,v) \mapsto g(s-u,t-v)\omega(u,v)$, which belongs to $L^2([-1,1]^2)$, with respect to $\wn$. Since the Wiener integral is a linear isometry between the integrand space $L^2([-1,1]^2)$ and the space $L^2(\Omega_\wn)$ of random variables (see,\ e.g.,\ \cite[pp.\ 7--8]{Nua1995} for details), no issues will arise with the measurability of $Y_{(s,t)}$. 

\revised{Let us briefly look into some of the probabilistic properties of the ambit field $Y$ (more details can be found in the survey article \cite{BBV2012}). Given $\sigma$, the field $Y$ is conditionally centered Gaussian with the conditional covariance function
\begin{equation*}
\big((s,t),(s',t')\big) \mapsto \E_W[Y_{(s,t)}Y_{(s',t')}]=\iint\limits_{A_g}  g(s'-s+u,t'-t+v)g(u,v) \sigma^2_{(s-u,t-v)} \ud u \ud v,
\end{equation*}
where $\E_W$ stands for expectation with respect to $\prob_W$.
Thus, $Y$ is non-stationary conditional on $\sigma$, unless almost any realization of $\sigma$ is a constant function. It is worth stressing that in many cases the one-parameter process $Y^{(s)}_t \eqdefl Y_{(s,t)}$, $t \in [0,1]$, where $s \in [0,1]$ is kept fixed, is not a semimartingale. In fact the following example shows that even a very simple, uniform weight function can result in a non-semimartingale.

\begin{exm}
Suppose that $g = \mathbf{1}_{[0,1]\times[0,y]}$ for some $y \in (0,1]$ and that $\sigma = 1$. By the finite additivity of the white noise $W$, it holds that
\begin{equation*}
Y^{(s)}_t = W([s-1,s]\times[t-y,t]) =  \widetilde{W}^{(s)}_{t+y} -\widetilde{W}^{(s)}_{t}, \quad t \in [0,1], 
\end{equation*}
where
\begin{equation*}
\widetilde{W}^{(s)}_{t} \eqdefl W([s-1,s]\times[-y,t-y]), \quad t \in [0,1+y],
\end{equation*}
is a Brownian motion. It follows from Example 5.7 of \cite{B2010}, up to a linear time change, that $Y^{(s)}$ is not a semimartingale when $y < 1$, and that the semimartingale property does in fact hold when $y = 1$.
\end{exm}

\begin{rem}
When $\sigma$ is a constant, the process $Y^{(s)}$ is stationary Gaussian and admits a moving average representation with respect to a Brownian motion, by a result of Karhunen \cite[Satz 5]{K1950}. To outline the argument, we can extend the process $Y^{(s)}$ to $\R$ by extending the driving white noise $W$ to $\R^2$. It follows from the continuity of translations in $L^2(\R^2)$ (see, e.g., \cite[p.\ 170]{HL1999}) and the isometry property of Wiener integrals that $Y^{(s)}_t \rightarrow Y^{(s)}_0$ in $L^2(\Omega)$ as $t \rightarrow 0$.
Moreover, since $A_g \subset [0,1]^2$, we have
\begin{equation*}
\bigcap_{u \in \R}\overline{\mathrm{span}} \big\{ Y^{(s)}_t : t \in (-\infty,u]\big\} \subset \bigcap_{u \in \R} \overline{\mathrm{span}} \{ W(E) : E \in \mathcal{B}([s-1,s]\times [u-1,u])\} = \{ 0 \},
\end{equation*}
where $\overline{\mathrm{span}}$ stands for the closed linear span in $L^2(\Omega)$. Theorem 2.5 of \cite{C2003}, which is a consequence of Satz 5 of \cite{K1950}, implies that there exist a weight function $\tilde{g} \in L^2\big((0,\infty)\big)$ and a standard Brownian motion $\big(\overline{W}_t\big)_{t \in \R}$ such that $Y^{(s)}$ equals in law to the moving average process
\begin{equation*}
\int_{-\infty}^t \tilde{g}(t-u) \ud \overline{W}_u, \quad t \in \R.
\end{equation*}

\end{rem}
}

\begin{rem}
Given any continuous function $\Gamma : [0,1] \longrightarrow [0,1]^2$, i.e., a \emph{curve}, we may define a stochastic process $(Y_{\Gamma(t)})_{t \in [0,1]}$, giving the description of the ambit field $Y$ as seen by an observer moving along the curve $\Gamma$. Such  processes are called \emph{ambit processes}. Barndorff-Nielsen and Graversen \cite{BNG2011} study the limit behavior of the quadratic variation of $(Y_{\Gamma(t)})_{t \in [0,1]}$ in the case where $\Gamma$ is a line segment, establishing sufficient conditions for the law of large numbers.
\end{rem}

\subsection{Power variation and concentration measure}\label{subs:pv}

For a two-parameter random field, an \emph{increment} is naturally defined over a \emph{rectangle} in the parameter space. Specifically, the rectangular increment of the ambit field $Y$ over $R \eqdefl (s_1,s_2] \times (t_1,t_2] \subset [0,1]^2$ is defined as
\begin{equation}\label{eq:planarincr}
Y(R) \eqdefl Y_{(s_2,t_2)}-Y_{(s_1,t_2)}-Y_{(s_2,t_1)}+Y_{(s_1,t_1)}.
\end{equation}
The definition \eqref{eq:planarincr} is standard in the literature of random fields, and can be recovered for example by partial differencing of $Y_{(s,t)}$ with respect to $s$ and $t$ --- or vice versa.
Although not needed in the sequel, it is worth pointing out the fact that the map $R \mapsto Y(R)$ can be extended to a finitely additive random measure on the algebra generated by finite unions and intersections of rectangles in $[0,1]^2$, which motivates the notation $Y(R)$.

For fixed $p>0$, we shall consider the $p$-th power variation of $Y$ over the square lattice $\mathcal{S}_n \eqdefl \big\{ \big(\frac{i}{n},\frac{j}{n}\big) : i,\, j = 0,1,\ldots, n\big\} \subset [0,1]^2$ for any $n \in \N$. Based on the values of $Y$ on the lattice $\mathcal{S}_n$, we may compute the increments of $Y$ over the rectangles
\begin{equation*}
R^{(n)}_{(i,j)} \eqdefl \big((i-1)/n,i/n \big] \times \big((j-1)/n,j/n \big], \quad \textrm{$i$,~$j = 1,\ldots,n$.}
\end{equation*}
Using them, we define the $p$-th power variation of $Y$ over $\mathcal{S}_n$ by
\begin{equation}\label{eq:pvariation}
V^{(p)}_{(s,t)}(k,n)\eqdefl\sum_{i=1}^{\lfloor ns/k \rfloor}\sum_{j=1}^{\lfloor nt/k \rfloor} \big|Y\big(R^{(n)}_{(ki,kj)}\big)\big|^p, \quad (s,t) \in [0,1]^2.
\end{equation}
where $k \in \N$ is a \emph{thinning parameter}. This allows us to take only every $k$-th increment into account when computing the power variation. The case $k=1$ corresponds to ordinary power variations whereas letting $k>1$ gives rise to thinned power variations. Note that we regard $V^{(p)}(k,n)$ as a random field on $[0,1]^2$.

To state the assumptions of our limit theorems, we need to introduce a technical device that controls the interdependence of the increments appearing in \eqref{eq:pvariation}. Let us first define $h_n \in L^2(\R^2)$ for any $n \in \N$ by
\begin{equation*}
h_n(s,t) \eqdefl  g(s,t) - g(s-1/n,t) - g(s,t-1/n) + g(s-1/n,t-1/n),
\end{equation*}
which, in fact, enables us to write succinctly
\begin{equation*}
Y\big(R^{(n)}_{(i,j)}\big) = \int h_n(i/n-u, j/n-v) \sigma_{(u,v)} \wn(\ud u, \ud v).
\end{equation*}
Since $g$ is non-vanishing, we have $c_n \eqdefl \int_{\R^2}h_n(\bs{z})^2 \ud \bs{z} \in (0,\infty)$. Thus, we may define $\pi_n \in \mathcal{P}(\R^2)$ by
\begin{equation*}
\pi_n(\ud \bs{z}) \eqdefl \dot{\pi}_n (\bs{z})\ud \bs{z}, \quad \textrm{where} \quad \dot{\pi}_n(\bs{z}) \eqdefl \frac{h_n(\bs{z})^2}{c_n}. 
\end{equation*}
The probability measure $\pi_n$ is a so-called \emph{concentration measure}, analogous to the ones appearing in earlier papers on ambit processes \cite[p.\ 265]{BNG2011} and Brownian semistationary processes \cite[p.\ 1166]{BCP2011}. Roughly speaking, the strength of the interdependence of the increments \eqref{eq:pvariation} is related to how dispersed $\pi_n$ is. Our limit theorems are based on the key assumption that the interdependence is not ``too strong'', in the sense that the sequence $\pi_1,\pi_2,\ldots$ converges weakly to a probability measure that is supported on a ``small'' subset of $\R^2$. 

\revised{
\begin{rem}\label{different-lags}
In addition to the square lattices $\mathcal{S}_n$, $n \in \N$, one could also consider observations of $Y$ on more general rectangular lattices $\mathcal{R}_n$, $n \in \N$, where
\begin{equation*}
\mathcal{R}_n \eqdefl \bigg\{ \bigg(\frac{i}{m^{(1)}_n},\frac{j}{m^{(2)}_n} \bigg) : i = 0,1,\ldots,m^{(1)}_n, \, j = 0,1,\ldots,m^{(2)}_n \bigg\}, 
\end{equation*}
and $\big(m^{(1)}_n\big)_{n \in \N}$,\, $\big(m^{(2)}_n\big)_{n \in \N} \subset \N$ are such that $m^{(1)}_n$,\, $m^{(2)}_n \rightarrow \infty$ as $n \rightarrow \infty$. The $p$-th power variation of $Y$ over $\mathcal{R}_n$ can be defined as
\begin{multline*}
\widetilde{V}^{(p)}_{(s,t)}\big(k^{(1)},k^{(2)},n\big)\\ \eqdefl\sum_{i=1}^{\lfloor m^{(1)}_n s/k^{(1)} \rfloor}\sum_{j=1}^{\lfloor m^{(2)}t/k^{(2)} \rfloor} \Bigg| Y \Bigg( \bigg(\frac{k^{(1)} i - 1}{m^{(1)}_n}, \frac{k^{(1)} i}{m^{(1)}_n}\bigg] \times \bigg(\frac{k^{(2)} j - 1}{m^{(2)}_n}, \frac{k^{(2)} j}{m^{(2)}_n}\bigg]
\Bigg)\Bigg|^p, \quad (s,t) \in [0,1]^2,
\end{multline*}
with two thinning parameters $k^{(1)}$, $k^{(2)} \in \N$. Moreover, the corresponding concentration measure $\tilde{\pi}_n$ is defined via a density that is the square of the function $\tilde{h}_n \in L^2(\R^2)$, given by
\begin{equation*}
\tilde{h}_n(s,t) \eqdefl  g(s,t) - g\big(s-1/m^{(1)}_n,t\big) - g\big(s,t-1/m^{(2)}_n\big) + g\big(s-1/m^{(1)}_n,t-1/m^{(2)}_n\big),
\end{equation*}
divided by $\tilde{c}_n \eqdefl \int_{\R^2} \tilde{h}_n (\bs{z})^2 \ud \bs{z}$.
\end{rem}
}

\subsection{Limit theorems}

We state now the main results of the paper. Their proofs, along with some auxiliary lemmas, are deferred to Sections \ref{sec:lln} and \ref{sec:clt}. In this section, $(k_n)_{n \in \N}$ stands for a \revised{fixed } non-decreasing sequence of natural numbers, which we shall use as the values of the thinning parameter, such that 
$\varepsilon_n \eqdefl k_n/n \rightarrow 0$. However, the assumption that $k_n \rightarrow \infty$ is not imposed yet.

Our first result is a functional law of large numbers for $V^{(p)}(k_n,n)$. The key assumption, which was alluded to above, behind the law of large numbers is the following.

\begin{asm}\label{asm:lln}
There exists $\pi \in \mathcal{P}(\R^2)$ such that $\lambda_2(\supp \, \pi)=0$ and $\pi_n \stackrel{w}{\rightarrow} \pi$.
\end{asm}

The condition $\lambda_2(\supp \, \pi)=0$ holds, for example, when $\pi$ is concentrated on a curve and, in particular, when $\pi$ is a convex combination of finitely many Dirac measures. Examples of weight functions $g$ that satisfy Assumption \ref{asm:lln} are given in Section \ref{subs:weight}.

In the statements below, $D([0,1]^2)\subset\R^{[0,1]^2}$ stands for the natural two-parameter generalization of the c\`adl\`ag space $D([0,1])\subset \R^{[0,1]}$. We endow this space with the uniform topology. Appendix \ref{app:skorohodspace} recalls the precise definition of $D([0,1]^2)$, along with some useful related facts.

\begin{thm}[Law of large numbers]\label{thm:powerlln}
If Assumption \ref{asm:lln} holds, then
\begin{equation*}
\frac{\varepsilon^2_n}{c^{p/2}_{n}}V^{(p)}(k_n,n) \xrightarrow[ n\rightarrow \infty]{\prob} m_p \Sigma^{(p,\pi)} \quad \textrm{in $D([0,1])^2$,}
\end{equation*}
where
\begin{equation*}
\Sigma^{(p,\pi)}_{(s,t)} \eqdefl \int_0^s \int_0^t \bigg(\int \sigma_{(u-\xi,v-\tau)}^2\pi(\ud \xi,\ud \tau)\bigg)^{p/2} \ud u \ud v, \quad (s,t) \in [0,1]^2.
\end{equation*}
\end{thm}

\begin{rem}Assumption \ref{asm:lln} is slightly more restrictive than mere \emph{mutual singularity} of $\pi$ and $\lambda_2$. Indeed, the proof of Theorem \ref{thm:powerlln} uses a separation argument that relies on the existence of a \emph{closed} $\lambda_2$-null set with full $\pi$-measure.
\end{rem}

The case where $\pi = \delta_{(s_0,t_0)}$, for some $(s_0,t_0) \in [0,1]^2$, is of particular interest. Then, we have
\begin{equation*}
\Sigma^{(p,\pi)}_{(s,t)} = \int_{-s_0}^{s-s_0} \int_{-t_0}^{t-t_0} \sigma_{(u,v)}^p \ud u \ud v.
\end{equation*}
From a practical point of view, the case where $\pi$ is \emph{not} a Dirac measure is somewhat undesirable. Then the random field $\Sigma^{(p,\pi)}$ ``sees'' merely a weighted space--time average of $\sigma$, and inferring the ``pure'' $\sigma$ may become impossible. 

Our second result is a functional central limit theorem for $V^{(p)}(k_n,n)$. Here, we concentrate on the case where $\pi$ is a Dirac measure and $k_n \rightarrow \infty$. For the needs of the central limit theorem, we refine Assumption \ref{asm:lln} by quantifying the speed of the convergence $\pi_n \stackrel{w}{\rightarrow} \pi$ as follows.

\begin{asm}\label{asm:clt}

There exist open sets $E_1,E_2,\ldots \subset \R^2$ and $\bs{z}_0 \eqdefl (s_0,t_0) \in [0,1]^2$ such that for all $n \in \N$,
\begin{enumerate}[label=(\roman*),ref=\roman*]
\item\label{c:null} $\bs{z}_0 \in \overline{E_n}$,
\item\label{c:inter} $\lambda_2 \big(E_n \cap (E_n+(s,t))\big)=0$ for any $(s,t) \in \R^2$ such that $|s|\vee |t| \geqslant \varepsilon_n$,
\item\label{c:decay} $\pi_n(\R^2 \setminus E_n ) = o(\varepsilon_n^2)$,
\end{enumerate}
\revised{where $\varepsilon_n = k_n/n$, as defined above.}
\end{asm}

The sets $E_1,E_2,\ldots$ should be seen as shrinking ``neighborhoods'' of the point $\bs{z}_0$.
In fact, items \eqref{c:null} and \eqref{c:inter} imply that for all $n\in \N$, 
\begin{equation*}
E_n \subset [s_0-\varepsilon_n,s_0+\varepsilon_n]\times[t_0-\varepsilon_n,t_0+\varepsilon_n].
\end{equation*}
Thus, by item \eqref{c:decay}, Assumption \ref{asm:lln} holds with $\pi = \delta_{\bs{z}_0}$. Concrete examples of specifications of the weight function $g$ that satisfy Assumption \ref{asm:clt} are provided in \eqref{eq:gspec} and \eqref{eq:gspec2}, below.

The central limit theorem is stated in terms of \emph{stable convergence in law}, a notion due to R\'enyi \cite{Ren1963}, which is the standard mode of convergence used in central limit theorems for power, bipower, and multipower variations of stochastic processes \revised{(see also \cite{AE1978} for more details on stable convergence). } For the convenience of the reader, we recall here the definition.

\begin{defn}[Stable convergence in law] Let $U_1,U_2,\ldots$ be random elements in a metric space $\mathcal{U}$, defined on the probability space $(\Omega,\mathcal{F},\prob)$, and let $U$ be a random element in $\mathcal{U}$, defined on $(\Omega',\mathcal{F}',\prob')$, an extension of $(\Omega,\mathcal{F},\prob)$. When $\mathcal{G}\subset\mathcal{F}$ is a $\sigma$-algebra, we say that $U_1,U_2,\ldots$ converge \emph{$\mathcal{G}$-stably in law} to $U$ and write $U_n \stackrel{L_\mathcal{G}}{\longrightarrow} U$, if 
\begin{equation}\label{eq:stable}
\E[f(U_n) V] \xrightarrow[ n\rightarrow \infty]{} \E'[f(U) V]
\end{equation}
for any bounded, $\mathcal{G}$-measurable random variable $V$ and bounded $f \in  C(\mathcal{U},\R)$.
\end{defn}

\begin{rem} Choosing $V = 1$ in \eqref{eq:stable} shows that stable convergence implies ordinary convergence in law. However, the converse is not true in general. 
\end{rem}

\begin{thm}[Central limit theorem]\label{thm:powerclt}
If Assumption \ref{asm:clt} holds, then
\begin{equation}\label{eq:powerlln}
\frac{\varepsilon_n}{c^{p/2}_n} \big(V^{(p)}(k_n,n) - \E_\wn\big[V^{(p)}_\cdot(k_n,n)\big] \big) \xrightarrow[ n\rightarrow \infty]{L_\mathcal{F}}  (m_{2p}-m^2_p)^{1/2} \Xi^{(p)} \quad \textrm{in $D([0,1]^2)$,}
\end{equation}
where
\begin{equation*}
\Xi^{(p)}_{(s,t)} \eqdefl \int\limits_{[-s_0,s-s_0]\times[-t_0,t-t_0]} \sigma^p_{(u,v)} \wn^\perp(\ud u, \ud v), \quad (s,t) \in [0,1]^2
\end{equation*}
and $W^{\perp}$ is a white noise on $[0,1]^2$ with control measure $\lambda_2$, independent of $\mathcal{F}$, defined on an extension of $(\Omega,\mathcal{F},\prob)$. 
\end{thm}

\revised{
\begin{rem}
Theorems \ref{thm:powerlln} and \ref{thm:powerclt} could be extended to the setting of Remark \ref{different-lags} as follows. Let us introduce two non-decreasing sequences $\big(k^{(1)}_n\big)_{n \in \N}$,\, $\big(k^{(2)}_n\big)_{n \in \N}\subset \N$ specifying the values of the thinning parameters, and define $\varepsilon^{(1)}_n \eqdefl k^{(1)}_n/m^{(1)}_n$, $\varepsilon^{(2)}_n \eqdefl k^{(2)}_n/m^{(2)}_n$, $n \in \N$. We assume that $\varepsilon^{(1)}_n$, $\varepsilon^{(2)}_n \rightarrow 0$ as $n \rightarrow \infty$. Provided that $\tilde{\pi}_n \stackrel{w}{\rightarrow} \pi$, where $\pi$ is as in Assumption \ref{asm:lln}, a law of large numbers holds for the random fields 
\begin{equation*}
\frac{\varepsilon^{(1)}_n\varepsilon^{(2)}_n}{\tilde{c}^{p/2}_n} \widetilde{V}^{(p)}\big(k^{(1)}_n,k^{(2)}_n,n\big), \quad n \in \N,
\end{equation*}
in $D([0,1]^2)$ with the limit given in Theorem \ref{thm:powerlln}. With regards to the central limit theorem, Assumption \ref{asm:clt} needs to be modified as follows. Condition \eqref{c:inter} is required to hold for any $(s,t) \in \R^2$ such that $|s| \geqslant \varepsilon^{(1)}_n$ or $|t|\geqslant \varepsilon^{(2)}_n$. Moreover, condition \eqref{c:decay} should be replaced with
$\tilde{\pi}_n(\R^2 \setminus E_n) = o \big(\varepsilon^{(1)}_n\varepsilon^{(2)}_n \big)$. Under Assumption \ref{asm:clt}, with these modifications, the random fields
\begin{equation*}
\frac{\sqrt{\varepsilon^{(1)}_n\varepsilon^{(2)}_n}}{\tilde{c}^{p/2}_n} \big(\widetilde{V}^{(p)}\big(k^{(1)}_n,k^{(2)}_n,n\big) - \E_\wn\big[\widetilde{V}^{(p)}_\cdot \big(k^{(1)}_n,k^{(2)}_n,n\big)\big] \big), \quad n \in \N,
\end{equation*}
satisfy a stable central limit theorem in $D([0,1]^2)$ with the limit given in Theorem \ref{thm:powerclt}. 
\end{rem}
}

\subsection{Weight functions}\label{subs:weight}

We shall now briefly discuss some examples of weight functions $g$ that satisfy Assumptions \ref{asm:lln} or \ref{asm:clt}.

\subsubsection{Uniform weight function}

Perhaps the simplest possible weight function is such that it assigns uniform weight over a rectangle.
More concretely, let
\begin{equation*}
g \eqdefl \mathbf{1}_{[s_1,s_2]\times[t_1,t_2]},
\end{equation*}
where $0 \leqslant s_1 < s_2 \leqslant 1$ and $0 \leqslant t_1 < t_2 \leqslant 1$. For any $n > 1/\big((s_2-s_1)\wedge (t_2-t_1)\big)$, we have
\begin{equation*}
\begin{split}
h_n &= \mathbf{1}_{[s_1,s_1+1/n]\times[t_1,t_1+1/n]} - \mathbf{1}_{[s_2,s_2+1/n]\times[t_1,t_1+1/n]}\\ & \quad - \mathbf{1}_{[s_1,s_1+1/n]\times[t_2,t_2+1/n]} +\mathbf{1}_{[s_2,s_2+1/n]\times[t_2,t_2+1/n]} \qquad \revised{\textrm{almost everywhere.}}
\end{split}
\end{equation*}
It is easy to check that then $c_n = 4/n$ and that Assumption \ref{asm:lln} holds
 with $\pi = (1/4)(\delta_{(s_1,t_1)}+\delta_{(s_2,t_1)}+\delta_{(s_1,t_2)}+\delta_{(s_2,t_2)})$. Thus, Theorem \ref{thm:powerlln} implies that
\begin{equation*}
V^{(2)}(1,n)\xrightarrow[ n\rightarrow \infty]{\prob} \int_0^\cdot \int_0^\cdot \big(\sigma^2_{(u-s_1,v-t_1)}+\sigma^2_{(u-s_2,v-t_1)}+\sigma^2_{(u-s_1,v-t_2)}+\sigma^2_{(u-s_2,v-t_2)}\big) \ud u \ud v
\end{equation*}
in $D([0,1]^2)$. Assumption \ref{asm:clt}, of course, cannot hold under this specification of $g$.

\subsubsection{Weight function with a singularity}

To satisfy Assumption \ref{asm:clt}, the weights imposed by $g$ should be concentrated to a neighborhood of some point in $[0,1]^2$. For example, let us consider $g \in L^2(\R^2)$ with a singularity at zero, given by
\begin{equation}\label{eq:gspec}
g(s,t) \eqdefl \begin{cases} (s \vee t)^{-\alpha}\ell(s \vee t), & (s,t) \in (0,1)^2,\\
0, & (s,t) \in \R^2 \setminus (0,1)^2,
\end{cases} 
\end{equation}  
where $\alpha \in (0,1)$ and $\ell \in C^1(0,1)$ is such that $\lim_{s \rightarrow 0+}\ell(s) \neq 0$, $\lim_{s \rightarrow 1-}\ell(s) = 0$, and $\|\ell'\|_\infty \eqdefl \sup_{s \in (0,1)}|\ell'(s)|<\infty$. Note that, necessarily, we have also $\| \ell \|_\infty \eqdefl \sup_{s \in (0,1)}|\ell(s)|<\infty$. A simple example of such a function is $\ell(s) \eqdefl 1-s$.

Assumption \ref{asm:clt} holds under this specification provided that the thinning parameter $k_n$ has suitably fast rate of growth. The following result gives a sufficient condition in terms of the asymptotic behavior of $\varepsilon_n$. Its proof is carried out in Section \ref{sec:singular}.

\begin{prop}[Weight function with a singularity]\label{prop:examplekernel}
Suppose that $g$ is given by \eqref{eq:gspec}.
\begin{enumerate}[label=(\arabic*),ref=\arabic*]
\item Assumption \ref{asm:lln} holds with $\pi = \delta_{\bs{0}}$,
\item If $\varepsilon_n \asymp n^{-\kappa}$, where $0 <\kappa \leqslant \alpha$ when $\alpha \in (0,1/2)$ and $0 < \kappa < (2 \alpha + 1)/(2\alpha +3)$ when $\alpha \in [1/2,1)$, then Assumption \ref{asm:clt} holds with $\bs{z}_0 = \bs{0}$ and $E_n = (0,\varepsilon_n)^2$. 
\end{enumerate} 
\end{prop}

\begin{rem}
If we assume further that $\lim_{s \rightarrow 1-}\ell'(s) = 0$, then it is possible to show that $0 < \kappa < (2 \alpha + 1)/(2\alpha +3)$ is a sufficient condition for all $\alpha \in (0,1)$.
\end{rem}

\subsubsection{Weight function supported on a triangle} As another example, let $\alpha\in (1/2,1)$ and $\ell$ as above, and define $g \in L^2(\R^2)$ through
\begin{equation}\label{eq:gspec2}
g(s,t) \eqdefl \begin{cases} t^{-\alpha}\ell(t), & (s,t) \in T,\\
0, & (s,t) \in \R^2 \setminus T,
\end{cases} 
\end{equation}
where
$T \eqdefl \{(s,t) : (1-t)/2 < s < (1+t)/2, \, 0 < t < 1 \}$  
is the isosceles triangle with vertices $(1/2,0)$, $(0,1)$, and $(1,1)$. Such a weight function is typical in space--time modeling of turbulence (see, e.g., \cite{BNES2005,BS2005,SCEPG2004}). Interpreting $s$ as a one-dimensional space variable and $t$ as time, the set $T$ (or more appropriately $-T$) can be seen as a \emph{causality cone}. 

Due to the different shape of the support, under this specification of $g$ we require that the thinning parameter grows at a faster rate compared to the preceding example. The proof of the following result is very similar to the one of Proposition \ref{prop:examplekernel}, so it is merely sketched in Section \ref{subsec:triangle}.

\begin{prop}[Weight function supported on a triangle]\label{prop:trianglekernel}
Suppose that $g$ is given by \eqref{eq:gspec2}. 
\begin{enumerate}[label=(\arabic*),ref=\arabic*]
\item Assumption \ref{asm:lln} holds with $\pi = \delta_{\bs{z}_0}$, where $\bs{z}_0 = (1/2,0)$.
\item If $\varepsilon_n \asymp n^{-\kappa}$, where $\kappa \in \big(0, (2\alpha-1)/(2\alpha+1)\big)$, then Assumption \ref{asm:clt} holds with  $E_n = (1/2-\varepsilon_n/2,1/2+\varepsilon_n/2)\times(0,\varepsilon_n/2)$.
\end{enumerate}
\end{prop}

\begin{rem}
It is evident from the proof that Proposition \ref{prop:trianglekernel} can be easily extended to a weight function $g$ whose essential support is a ``small perturbation'' of the triangle $T$.
\end{rem}

\subsection{Some comments on the results}\label{sec:discuss}

\subsubsection{Measurement of relative volatility}

A practical difficulty in using Theorem \ref{thm:powerlln} is that the power variations need to be scaled appropriately and the scaling depends on the unknown weight function $g$ and may be difficult to compute precisely. In fact, it is evident that the volatility field $\sigma$ cannot even be determined unambiguously unless $g$ normalized a priori. However, often we are more interested in the \emph{variation} $\sigma$ rather than its precise \emph{level}, which may not, thus, be very informative due to the ambiguity caused by the lack of normalization. It is key to note that the variation of $\sigma$ is captured also by the \emph{relative} integrated volatility field
\begin{equation}\label{eq:relintvol}
\frac{\int_0^s \int_0^t \sigma^2_{(u,v)}\ud u \ud v}{\int_0^1 \int_0^1 \sigma^2_{(u,v)}\ud u \ud v}, \quad (s,t) \in [0,1]^2.
\end{equation}  
Quantities of the form \eqref{eq:relintvol} can be obtained as the limits of certain ratios of (unscaled) power variations, which are statistically feasible. More precisely, Theorem \ref{thm:powerlln} readily implies that when $\pi = \delta_{(0,0)}$, we have for any $p>0$,
\begin{equation*}
\frac{V^{(p)}_\cdot(1,n)}{V^{(p)}_{(1,1)}(1,n)} \xrightarrow[ n\rightarrow \infty]{\prob} \frac{\int_0^\cdot \int_0^\cdot \sigma^p_{(u,v)}\ud u \ud v}{\int_0^1 \int_0^1 \sigma^p_{(u,v)}\ud u \ud v} \quad \textrm{in $D([0,1]^2)$.}
\end{equation*}
The use of relative volatility statistics, in general, is elaborated in the paper \cite{BNG2013}.

\subsubsection{Bias in the central limit theorem}

Note that in Theorem \ref{thm:powerclt}, the scaled power variation $\varepsilon^2_n c^{-p/2} V^{(p)}(k_n,n)$ is centered around its expectation $\varepsilon^2_n c^{-p/2}\E_\wn\big[V^{(p)}_\cdot(k_n,n)\big]$, instead of the limit $m_p \Sigma^{(p,\pi)}$ given by the law of large numbers. While it is shown in the proof of Theorem \ref{thm:powerlln} that, under Assumption \ref{asm:lln},
\begin{equation}\label{eq:bias}
\frac{\varepsilon^2_n}{c^{p/2}_n}\E_\wn\big[V^{(p)}_{(s,t)}(k_n,n)\big] \xrightarrow[ n\rightarrow \infty]{} m_p \Sigma^{(p,\pi)}_{(s,t)} \quad \textrm{for any $(s,t) \in [0,1]^2$,}
\end{equation}
the rate of convergence in \eqref{eq:bias} appears to be in most, if not all, cases too slow that we could replace $\varepsilon^2_n c^{-p/2}\E_\wn\big[V^{(p)}_\cdot(k_n,n)\big]$ with $m_p \Sigma^{(p,\pi)}$ in \eqref{eq:powerlln}.

An asymptotically non-negligible bias is present even in the most well-behaved case with constant $\sigma$. Namely, we have then
\begin{equation*}
\frac{\varepsilon^2_n}{c^{p/2}_n}\E_\wn\big[V^{(p)}_{(s,t)}(k_n,n)\big] - m_p \Sigma^{(p,\pi)}_{(s,t)} = -m_p \varepsilon_n \bigg( \bigg\{\frac{s}{\varepsilon_n} \bigg\}t+\bigg\{\frac{t}{\varepsilon_n} \bigg\}s +o(1)\bigg).
\end{equation*}
One can show that for almost any $(s,t) \in [0,1]^2$,
\begin{equation*}
\limsup_{n \rightarrow \infty} \bigg( \bigg\{\frac{s}{\varepsilon_n} \bigg\}+\bigg\{\frac{t}{\varepsilon_n} \bigg\}\bigg) > 0,
\end{equation*}
and, consequently,
\begin{equation*}
\liminf_{n \rightarrow \infty} \varepsilon^{-1}_n \bigg(\frac{\varepsilon^2_n}{c^{p/2}_n}\E_\wn\big[V^{(p)}_{(s,t)}(k_n,n)\big] - m_p \Sigma^{(p,\pi)}_{(s,t)}\bigg) < 0.
\end{equation*}
This peculiarity limits the usefulness of Theorem \ref{thm:powerclt} in the context of statistical inference (e.g., regarding confidence intervals) on $\Sigma^{(p,\pi)}$.

\subsubsection{Extending the central limit theorem beyond thinned power variations}

Is it possible to extend Theorem \ref{thm:powerclt} to cover \emph{ordinary} power variations? Quite possibly, but we expect that the limit would not remain the same. In fact, we conjecture that the situation is analogous to Brownian semistationary ($\mathcal{BSS}$) processes (see \cite{CHPP2012}). Recall that ordinary power variations of $\mathcal{BSS}$ processes, under certain conditions, satisfy a central limit theorem \cite[Theorem 3.2]{CHPP2012} with a limit analogous to $\Xi^{(p)}$, but multiplied with a constant that is strictly larger than $(m_{2p}-m^2_p)^{1/2}$, whereas the limit in the corresponding result for thinned power variations \cite[Theorem 4.5]{CHPP2012} has the factor $(m_{2p}-m^2_p)^{1/2}$. This is a consequence of the non-generate limiting correlation structure (which identical to the one of \revised{\emph{fractional Gaussian noise}}) of the increments of a $\mathcal{BSS}$ process. 
Thinning decreases the asymptotic variance in the central limit theorem through ``decorrelation'' of the increments, but at the expense of rate of convergence.   

While our Theorem \ref{thm:powerclt} is analogous to Theorem 4.5 of \cite{CHPP2012}, obtaining a central limit theorem for unthinned power variations, akin to Theorem 3.2 of \cite{CHPP2012}, is currently an open problem, which we hope to address in future work, along with allowing for $\sigma$ that depends on the driving noise. 
The key problem is the identification of the limiting correlation structure of the increments.
However, it seems that such a result cannot be accomplished by a straightforward modification of the arguments in \cite{BCP2011} since the one-dimensional regular variation techniques used with $\mathcal{BSS}$ processes appear unapplicable in our setting due to the additional dimension. We also expect that, like in \cite{BCP2011,BCP2012,CHPP2012}, such a result would require stronger assumptions on the dependence structure of the ambit field --- beyond what we formulate using the concentration measures --- and a smoothness condition on $\sigma$. 

\section{Law of large numbers}\label{sec:lln}

In this section, we prove the law of large numbers for power variations, Theorem \ref{thm:powerlln}.
 The proof is based on the conditional Gaussianity of the ambit field $Y$ given $\sigma$ and, in particular, on a covariance bound for nonlinear transformations of jointly Gaussian random variables, which we will review first. Note that $Y$ conditional on $\sigma$ is typically non-stationary and the existing laws of large numbers for Gaussian random fields appear not to be (at least directly) applicable to this setting.

\subsection{Hermite polynomials and a covariance bound}

Recall that the \emph{Hermite polynomials} $H_0, H_1, H_2,\ldots$ on $\R$ are uniquely defined through the generating function
\begin{equation*}
\exp\bigg(tx - \frac{t^2}{2} \bigg) = \sum_{n=1}^\infty t^n H_n(x), \quad x\in \R.
\end{equation*}
They are orthogonal polynomials with respect to the Gaussian measure $\gamma$ on $\R$. More precisely, if $(X_1,X_2)$ is a Gaussian random vector such that $\E[X_1] = \E[X_2] = 0$ and $\E[X_1^2]=\E[X_2^2]=1$, then (cf.\ \cite[Lemma 1.1.1]{Nua1995})
\begin{equation}\label{eq:hermiteort}
n!\E[H_n(X_1)H_m(X_2)] = \begin{cases} \E[X_1X_2]^n, & n=m,\\ 0, & n\neq m.
\end{cases}
\end{equation} 
Thus, $\big\{\sqrt{n!}H_n : n \in \N_0\big\}$ is an orthonormal basis of $L^2(\R,\gamma)$ and, in particular, for any $f \in L^2(\R,\gamma)$ there exists $(\alpha_0, \alpha_1, \ldots)\in \ell^2(\N_0)$ such that
\begin{equation}\label{eq:hermiteexpansion}
f = \sum_{n=0}^\infty \alpha_n \sqrt{n!} H_n \quad \textrm{in $L^2(\R,\gamma)$.}
\end{equation}
The index of the leading non-zero coefficient in  the expansion \eqref{eq:hermiteexpansion}, that is, $\min \{ k \in \N_0 : \alpha_k \neq 0\}$, is known as the \emph{Hermite rank} of the function $f$. 

Using \eqref{eq:hermiteort} and \eqref{eq:hermiteexpansion}, it is straightforward to establish the following bound for covariances of functions of jointly Gaussian random variables that is sometimes attributed to J.\ Bretagnolle (see, e.g., \cite[Lemme 1]{Guy1987}). This simple inequality is, in fact, a special case of a far more general result due to Taqqu \cite[Lemma 4.5]{Taqqu1977}. 

\begin{lem}[Covariance]\label{lem:hermiterank}
Let $(X_1,X_2)$ be as above. If $f \in L^2(\R,\gamma)$ has Hermite rank $r \in \N$, then
\begin{equation*}
|\E[f(X_1)f(X_2)]| \lesssim_{\revised{f,r}} |\E[X_1X_2]|^q \quad \textrm{for any $q \in [0,r]$.}
\end{equation*}
\end{lem}

For any $p>0$, write $u_p(x) \eqdefl |x|^p-m_p$, $x \in \R$. Clearly, $u_p \in L^2(\R,\gamma)$ and Gaussian integration by parts shows that the Hermite rank of $u_p$ is $2$. Thus, Lemma \ref{lem:hermiterank} implies that 
\begin{equation}\label{eq:powercov}
|\Cov[|X_1|^p,|X_2|^p]| \lesssim_p |\E[X_1X_2]|^q \quad \textrm{for any $q \in [0,2]$,}
\end{equation}
which will be instrumental in the proof of Theorem \ref{thm:powerlln}, below.

\subsection{Proof of Theorem \ref{thm:powerlln}}

Prior to proving Theorem \ref{thm:powerlln}, we still need to establish a simple fact that follows from the convergence $\pi_n \stackrel{w}{\rightarrow} \pi$. To this end, recall that the \emph{L\'evy--Prohorov distance} of $\mu$,\ $\nu \in \mathcal{P}(\R^2)$ is defined as
\begin{equation*}
d(\mu,\nu) \eqdefl \inf \big\{ \varepsilon > 0 : \mu(E) \leqslant \nu(E^\varepsilon) + \varepsilon, \, \nu(E) \leqslant \mu(E^\varepsilon) + \varepsilon  \textrm{ for all $E \in \mathcal{B}(\R^2)$}\big\}.
\end{equation*}  
The L\'evy--Prohorov distance is a metric on $\mathcal{P}(\R^2)$ and $\pi_n \stackrel{w}{\rightarrow} \pi$ holds if and only if $d(\pi_n,\pi)\rightarrow 0$ (see, e.g., \cite[p.~72]{Bil1999}). Below, we write $B \eqdefl \supp \, \pi$, for the sake of brevity.

\begin{lem}[Concentration]\label{lem:prohorov}
If $\pi_n \stackrel{w}{\rightarrow} \pi$, then there exist positive numbers $(a_n)$ such that $a_n \downarrow 0$ and $\pi_n(B^{a_n})\rightarrow 1$.
\end{lem}

\begin{proof}
Let $(a_n)$ be such that $a_n \downarrow 0$ and $a_n > d(\pi_n,\pi)$ for any $n\in\N$. By the definition of the L\'evy--Prohorov distance, $\pi(B) \leqslant \pi_n(B^{a_n}) + a_n$ for any $n \in \N$. 
Since $\pi(B)=1$, we have
$\pi_n(B^{a_n}) \geqslant  1-a_n \rightarrow 1$.
\end{proof}

\begin{proof}[Proof of Theorem \ref{thm:powerlln}]
Clearly, we have $V^{(p)}_{(s,t)}(k_n,n) \leqslant V^{(p)}_{(u,v)}(k_n,n)$ if $s \leqslant u$ and $t \leqslant v$. Thus, by Lemma \ref{lem:unifconv}, it suffices to establish pointwise convergence
\begin{equation}\label{eqn:pointwiseconv}
\varepsilon_n^2c^{-p/2}_{n}V^{(p)}_{(s,t)}(k_n,n) \xrightarrow[ n\rightarrow \infty]{\prob} m_p \int_0^{s} \int_0^{t} \bigg(\int \sigma_{(u-\xi,v-\tau)}^2\pi(\ud \xi,\ud \tau)\bigg)^{p/2} \ud u \ud v
\end{equation}
for any $(s,t)\in [0,1]^2$. More precisely, we show \eqref{eqn:pointwiseconv} \emph{conditional} on the realization of $\sigma$. Under this conditioning, we may regard $\sigma$ as a non-random element of $C([-1,1]^2,\R_+)$ and $Y$ as a Gaussian random field.

Let us first show that 
\begin{equation}\label{eq:meanconv}
\lim_{n\rightarrow \infty}\varepsilon_n^2c^{-p/2}_{n}\E_\wn\big[V^{(p)}_{(s,t)}(k_n,n)\big] = m_p \int_0^{s} \int_0^{t} \bigg(\int \sigma_{(u-\xi,v-\tau)}^2\pi(\ud \xi,\ud \tau)\bigg)^{p/2} \ud u \ud v.
\end{equation}
 Since
\begin{equation*}
\begin{split}
\E_\wn\big[\big|Y\big(R^{(n)}_{(i,j)}\big)\big|^p\big] & = m_p \E_\wn\big[\big|Y\big(R^{(n)}_{(i,j)}\big)\big|^2\big]^{p/2} \\
& = m_p c_n^{p/2} \bigg(\int \sigma^2_{(i/n - \xi, j/n - \tau)} \pi_n(\ud \xi,\ud \tau)\bigg)^{p/2},
\end{split} 
\end{equation*}
we have
\begin{equation*}
\begin{split}
\varepsilon_n^2c^{-p/2}_{n}\E_\wn\big[V^{(p)}_{(s,t)}(k_n,n)\big] & = m_p\varepsilon^2_n  \sum_{i=1}^{\lfloor s/\varepsilon_n \rfloor}\sum_{j=1}^{\lfloor t/\varepsilon_n \rfloor}  \bigg(\int \sigma^2_{(\varepsilon_n i - \xi, \varepsilon_n j - \tau)} \pi_n(\ud \xi,\ud \tau)\bigg)^{p/2} \\
& = m_p \int_0^{\lfloor s  \rfloor_n}\int_0^{\lfloor t \rfloor_n}\bigg(\int \sigma^2_{(\lceil u \rceil_n - \xi,  \lceil v \rceil_n - \tau)} \pi_n(\ud \xi,\ud \tau)\bigg)^{p/2}\ud u \ud v,
\end{split}
\end{equation*}
where $\lceil x \rceil_n \eqdefl \varepsilon_n\lceil x/\varepsilon_n \rceil$ and $\lfloor x \rfloor_n \eqdefl \varepsilon_n\lfloor x/\varepsilon_n \rfloor$ for any $x \in \R$ and $n \in \N$.
Since $\lfloor s \rfloor_n \rightarrow s$ and $\lfloor t  \rfloor_n \rightarrow t$ as $n \rightarrow \infty$, the convergence \eqref{eq:meanconv} follows from Lebesgue's dominated convergence theorem, provided that for any $(u,v) \in [0,1]^2$,
\begin{equation*}
\lim_{n \rightarrow \infty} \int \sigma^2_{(\lceil u \rceil_n \wedge 1 - \xi, \lceil v \rceil_n \wedge 1 - \tau)} \pi_n(\ud \xi,\ud \tau) = \int \sigma_{(u-\xi,v-\tau)}^2\pi(\ud \xi,\ud \tau),
\end{equation*}
which, in turn, is a straightforward consequence of the uniform continuity of the realization of $\sigma$ and the convergence $\pi_n \stackrel{w}{\rightarrow} \pi$.

Now, \eqref{eqn:pointwiseconv} follows from Chebyshev's inequality, provided that 
\begin{equation}\label{eqn:varconv}
\lim_{n\rightarrow \infty}\varepsilon^4_n c^{-p}_n\Var_\wn\big[V^{(p)}_{(s,t)}(k_n,n)\big]=0.
\end{equation}
To show \eqref{eqn:varconv}, we expand
\begin{multline}\label{eq:covexpansion}
\varepsilon^4_n c^{-p}_n\Var_\wn\big[V^{(p)}_{(s,t)}(k_n,n)\big] \\ = \varepsilon^4_n \sum_{i_1,i_2=1}^{\lfloor s/\varepsilon_n\rfloor}\sum_{j_1,j_2=1}^{\lfloor t/\varepsilon_n\rfloor} c^{-p}_{n} \Cov_\wn   \big[\big|Y\big(R^{(n)}_{(k_n i_1,k_n j_1)}\big)\big|^p,\big|Y\big(R^{(n)}_{(k_n i_2,k_n j_2)}\big)\big|^p\big].
\end{multline}
Using the inequality \eqref{eq:powercov} and the relation
\begin{equation}\label{eq:cequiv}
\E_\wn \big[\big|Y\big(R^{(n)}_{(i,j)}\big)\big|^2\big] \asymp_\sigma c_n, \quad i,j = 1,\ldots,n,\quad n  \in \N,
\end{equation}
we obtain
\begin{multline*}
c^{-p}_{n} \big|\Cov_\wn   \big[\big|Y\big(R^{(n)}_{(k_n i_1,k_n j_1)}\big)\big|^p,\big|Y\big(R^{(n)}_{(k_n i_2,k_n j_2)}\big)\big|^p\big]\big| \\
\begin{aligned}
& \lesssim_{\sigma,p} c_n \big|\E_\wn \big[Y\big(R^{(n)}_{(k_n i_1,k_n j_1)}\big)Y\big(R^{(n)}_{(k_n i_2,k_n j_2)}\big)\big]\big| \\
& \lesssim_\sigma \int \dot{\pi}_n(\xi,\tau)^{1/2}\dot{\pi}_n\big(\xi+\varepsilon_n (i_1-i_2),\tau+\varepsilon_n (j_1-j_2)\big)^{1/2}\ud \xi \ud \tau.
\end{aligned}
\end{multline*}
Applying this bound to \eqref{eq:covexpansion}, we arrive at
\begin{equation*}
\varepsilon^4_n c^{-p}_n\Var_\wn\big[V^{(p)}_{(s,t)}(k_n,n)\big]  \lesssim_{\sigma,p}  \int\limits_0^s \int\limits_0^t \int\limits_0^s \int\limits_0^t \Pi_n(u_1,v_1,u_2,v_2)\ud u_1 \ud v_1\ud u_2 \ud v_2,
\end{equation*}
where 
\begin{equation*}
\Pi_n(u_1,v_1,u_2,v_2) \eqdefl\int \dot{\pi}_n(\xi,\tau)^{1/2}\dot{\pi}_n\big(\xi+\lceil u_1\rceil_n-\lceil u_2\rceil_n,\tau+\lceil v_1\rceil_n-\lceil v_2\rceil_n \big)^{1/2}\ud \xi \ud \tau.
\end{equation*}

The Cauchy--Schwarz inequality ensures that the functions $\Pi_1,\Pi_2,\ldots$ are uniformly bounded on $([0,s]\times[0,t])^2$. Thus, by Lebesgue's dominated convergence theorem, it suffices to show that $\Pi_n$ tends to zero almost everywhere as $n\rightarrow \infty$. We will split this task into two parts by treating separately
\begin{equation*}
\Pi^{(1)}_n(u_1,v_1,u_2,v_2) \eqdefl  \int_{\R^2\setminus B^{a_n}} \dot{\pi}_n(\xi,\tau)^{1/2}\dot{\pi}_n\big(\xi+\lceil u_1\rceil_n-\lceil u_2\rceil_n,\tau+\lceil v_1
\rceil_n-\lceil v_2\rceil_n\big)^{1/2}\ud \xi \ud \tau
\end{equation*}
and
\begin{equation*}
\Pi^{(2)}_n(u_1,v_1,u_2,v_2) \eqdefl  \int_{B^{a_n}} \dot{\pi}_n(\xi,\tau)^{1/2}\dot{\pi}_n\big(\xi+\lceil u_1\rceil_n-\lceil u_2\rceil_n,\tau+\lceil v_1
\rceil_n-\lceil v_2\rceil_n\big)^{1/2}\ud \xi \ud \tau,
\end{equation*}
where $(a_n)$ is a sequence of positive real numbers such that $a_n\downarrow0$ and $\pi_n(B^{a_n})\rightarrow 1$, the existence of which is ensured by Lemma \ref{lem:prohorov}.
Applying the Cauchy--Schwarz inequality to $\Pi^{(1)}$, we obtain
\begin{equation*}
\begin{split}
 \Pi^{(1)}_n(u_1,v_1,u_2,v_2)^2 & \leqslant \pi_n(\R^2\setminus B^{a_n}) \int_{\R^2\setminus B^{a_n}}\dot{\pi}_n\big(\xi+\lceil u_1\rceil_n-\lceil u_2\rceil_n,\tau+\lceil v_1
\rceil_n-\lceil v_2\rceil_n\big)\ud \xi \ud \tau \\
& \leqslant 1-\pi_n(B^{a_n}) \xrightarrow[n \rightarrow \infty]{} 0.
\end{split}
\end{equation*}
Similarly, in the case of $\Pi^{(2)}$ we obtain
\begin{equation}\label{eq:pi2ineq}
\Pi^{(2)}_n(u_1,v_1,u_2,v_2)^2 \leqslant \int_{B^{a_n}}\dot{\pi}_n\big(\xi+\lceil u_1\rceil_n-\lceil u_2\rceil_n,\tau+\lceil v_1
\rceil_n-\lceil v_2\rceil_n\big) \ud \xi \ud \tau,
\end{equation}
where, however, a slightly more elaborate argument, inspired by the proof of Lemma 1 in \cite{BNG2011}, is needed to show convergence to zero.

By Urysohn's lemma, for any $\delta>0$ there exists $\varphi_{\delta} \in  C(\R^2,[0,1])$ such that $\varphi_\delta(\bs{z})=1$ for $\bs{z} \in \overline{B^{\delta}}$ and $\varphi_\delta(\bs{z})=0$ for $\bs{z} \in \overline{\R^2 \setminus B^{2\delta}}$. From \eqref{eq:pi2ineq} we deduce, thus,
\begin{equation*}
\begin{split}
\limsup_{n\rightarrow\infty}\Pi^{(2)}_n(u_1,v_1,u_2,v_2)^2 & \leqslant \lim_{n\rightarrow\infty}\int \varphi_{\delta}(\xi,\tau) \dot{\pi}_n\big(\xi+\lceil u_1\rceil_n-\lceil u_2\rceil_n,\tau+\lceil v_1
\rceil_n-\lceil v_2\rceil_n\big) \ud \xi \ud \tau \\
& = \lim_{n\rightarrow\infty}\int \varphi_{\delta}\big(\xi+\lceil u_2\rceil_n-\lceil
u_1\rceil_n,\tau+\lceil v_2
\rceil_n-\lceil v_1\rceil_n\big) \pi_n(\ud \xi, \ud \tau) \\
& = \int  \varphi_{\delta}(\xi+u_2-u_1,\tau+v_2-v_1) \pi(\ud \xi,\ud \tau),
\end{split}
\end{equation*}
where we used the bound $|\lceil x \rceil_n - x|< \varepsilon_n$, for all $x \in \R$ and $n \in \N$, and the observation that $\varphi_\varepsilon$ is, in fact, uniformly continuous. Since $\varphi_\delta$ converges pointwise to $\mathbf{1}_B$ as $\delta \rightarrow 0$, we have
\begin{equation*}
\limsup_{n\rightarrow\infty}\Pi^{(2)}_n(u_1,v_1,u_2,v_2)^2 \leqslant \pi\big(\supp \, \pi + (u_1-u_2,v_1-v_2) \big).
\end{equation*}
The push-forward measure of the mapping $(u_1,v_1,u_2,v_2)\mapsto (u_1-u_2,v_1-v_2)$ on $\R^2$ is absolutely continuous with respect to $\lambda_2$, so our argument is complete if we show that $\pi(\supp \, \pi-\bs{z})=0$ for almost every $\bs{z}\in \R^2$. But this follows from the assumption that $\lambda_2(\supp\,\pi)=0$, since
\begin{equation*}
\int \pi(\supp\,\pi-\bs{z}) \lambda_2(\ud \bs{z})  = \int \lambda_2(\supp\,\pi-\bs{z})\pi(\ud \bs{z}) = \int \lambda_2(\supp\,\pi)\pi(\ud \bs{z}) = 0,
\end{equation*}
where the first equality follows from Lemma 1.28 in \cite{Kal2002} and the second from the translation invariance of the Lebesgue measure.
\end{proof}

\section{Central limit theorem}\label{sec:clt}

The proof of the central limit theorem, Theorem \ref{thm:powerclt}, is based on a \emph{chaos decomposition} of the power variation, that is, representing it as an $L^2$-convergent series of \emph{iterated Wiener integrals} with respect to the white noise $W$. Then, we apply the limit theory for iterated Wiener integrals to establish convergence of finite-dimensional distributions. We will begin by recalling some key facts of the chaos decomposition and the related central limit theorem.

\subsection{Central limit theorem via chaos decompositions}

Let us denote by $\Hil$ the Hilbert space $L^2([-1,1]^2)$, which will have a special role in what follows. Moreover, let $\Hil^{\otimes k} \cong L^2([-1,1]^{2k})$ be the $k$-fold tensor product of  $\Hil$, for $k \in \N$, and denote by $\Hil^{\odot k}$ the set of \emph{symmetric} functions belonging to $\Hil^{\otimes k}$, that is, for any $f \in \Hil^{\odot k}$, permutation $s : \{1,\ldots,k \}\longrightarrow \{1,\ldots,k \}$, and almost any $(\bs{z}_1,\ldots,\bs{z}_k) \in [0,1]^{2k}$,
\begin{equation*}
f(\bs{z}_1,\ldots,\bs{z}_1) = f\big(\bs{z}_{s(1)},\ldots,\bs{z}_{s(k)}\big).
\end{equation*}
For any $f \in \Hil^{\odot k}$, the $k$-fold \emph{iterated Wiener integral} of the kernel $f$ with respect to the white noise $\wn$, denoted by $I_k(f)$, can be defined as a linear map $\Hil^{\odot k} \longrightarrow L^2(\Omega_\wn)$ with the key property
\begin{equation*}
\E_\wn\big[I_k(f)^2\big] = k!\| f\|^2_{\Hil^{\otimes k}}.
\end{equation*}
(For the details of the construction, see \cite[pp.~7--10]{Nua1995}.) The remarkable feature of these integrals is that any $X\in L^2(\Omega_\wn)$ admits a unique \emph{chaos decomposition} \cite[Theorem 1.1.2]{Nua1995},
\begin{equation}\label{eq:chaos}
X = \sum_{k=0}^\infty I_k(f_k) \quad \textrm{in $L^2(\Omega_\wn)$,}
\end{equation}
where $f_k \in \Hil^{\odot k}$ for any $k \in \N_0$, with the convention that $f_0 \eqdefl \E_\wn[X]$ and $I_0$ is the identity map on $\R$.

If we are given a sequence of random variables in $L^2(\Omega_\wn)$ and we would like to show that they converge in law to a Gaussian distribution, the chaos decomposition \eqref{eq:chaos} turns out to be instrumental. Specifically, such convergence can be established by verifying some straightforward criteria on the associated kernel functions.
To formulate the criteria, recall that for any $r \in \{1,\ldots,k-1\}$, the $r$-th \emph{contraction} of $f^{(1)} = f^{(1)}_1\otimes \cdots \otimes f^{(1)}_k \in \Hil^{\otimes k}$ and $f^{(2)} = f^{(2)}_1\otimes \cdots \otimes f^{(2)}_k\in \Hil^{\otimes k}$ is the function 
\begin{equation*}
f^{(1)} \otimes_r f^{(2)} \eqdefl \prod_{i=1}^r \big\langle f^{(1)}_{k-r+i}, f^{(2)}_i\big\rangle_\Hil  f^{(1)}_1 \otimes \cdots f^{(1)}_{k-r} \otimes f^{(2)}_{r+1} \otimes \cdots \otimes f^{(2)}_k \in \Hil^{\otimes 2(k-r)}.
\end{equation*}
The following multivariate central limit theorem is a slight reformulation of Theorem 1 in \cite{BCPW2009}, originally a corollary of the results of Nualart and Peccati \cite{NP2005}, and Peccati and Tudor \cite{PT2005}.

\begin{lem}[CLT via chaos decompositions]\label{lem:chaosclt} Let $d \in \N$ and for any $n \in \N$, let $X^{(n)}_1,\ldots,X^{(n)}_d \in L^2(\Omega_\wn)$ be such that for any $i = 1,\ldots,d$,
\begin{equation*}
X^{(n)}_i = \sum_{k=1}^{\infty} I_k\big(f^{(n)}_{k,i}\big) \quad \textrm{in $L^2(\Omega_\wn)$,}
\end{equation*}
where $f^{(n)}_{k,i} \in \Hil^{\odot k}$. Suppose that
\begin{enumerate}[label=(\arabic*),ref=\arabic*]
\item for any $i = 1,\ldots,d$,
\begin{equation*}
\lim_{m \rightarrow \infty} \limsup_{n\rightarrow\infty}\sum_{k=m}^{\infty} k!\big\|f^{(n)}_{k,i}\big\|^2_{\Hil^{\otimes k}} = 0,
\end{equation*}
\item there exist positive semidefinite $d \times d$-matrices $\bs{\aleph},\bs{\aleph}^{(1)},\bs{\aleph}^{(2)},\ldots$ such that for any $i$,~$j = 1,\ldots,d$ and $k \in \N$,
\begin{equation*}
\lim_{n \rightarrow \infty} k!\big\langle f^{(n)}_{k,i},f^{(n)}_{k,j} \big\rangle_{\Hil^{\otimes k}} = \bs{\aleph}^{(k)}_{i,j},
\end{equation*}
and that $\sum_{k=1}^\infty \bs{\aleph}^{(k)} = \bs{\aleph}$,
\item  for any $i = 1,\ldots,d$, $k\in \N$, and $r = 1,\ldots,k-1$,
\begin{equation*}
\lim_{n \rightarrow \infty} \big\| f^{(n)}_{k,i} \otimes_r f^{(n)}_{k,i} \big\|^2_{\Hil^{\otimes 2(k-r)}} = 0.
\end{equation*}
\end{enumerate}
Then, $\big(X^{(n)}_1,\ldots,X^{(n)}_d\big) \stackrel{L}{\rightarrow} N_d(\bs{0},\, \bs{\aleph})$ as $n \rightarrow \infty$.
\end{lem}

\subsection{Proof of Theorem \ref{thm:powerclt}}

Throughout this section, apart from formula \eqref{eq:testfunction} below, we work conditional on the realization of $\sigma$, regarding it as deterministic --- similarly to the earlier proof of Theorem \ref{thm:powerlln}.
 
We introduce some convenient notation. We define for any $n \in \N$, and $i$,\ $j = 1,\ldots,\lfloor \varepsilon^{-1}_n \rfloor$, function $f_{n,(i,j)} \in \Hil$ by
\begin{equation*}
f_{n,(i,j)}(s,t) \eqdefl h_n(\varepsilon_n i-s, \varepsilon_n j-t)\sigma_{(s,t)},
\end{equation*}
and its normalized counterpart
$
\bar{f}_{n,(i,j)} \eqdefl \|f_{n,(i,j)}\|_\Hil^{-1} f_{n,(i,j)} \in \Hil
$.
By definition,
$
Y\big(R^{(n)}_{(k_n i,k_n j)} \big) = I_1(f_{n,(i,j)})
$.
Thus, we have
\begin{equation*}
\E_\wn\big[Y\big(R^{(n)}_{(k_n i,k_n j)}\big)^2\big] = \| f_{n,(i,j)}\|_\Hil^2
\end{equation*}
and
\begin{equation}\label{eq:incrcorr}
\Corr_W\big[Y\big(R^{(n)}_{(k_n i_1,k_n j_1)} \big),Y\big(R^{(n)}_{(k_n i_2,k_n j_2)} \big)\big] = \big\langle \bar{f}_{n,(i_1,j_1)}, \bar{f}_{n,(i_2,j_2)} \big\rangle_\Hil,
\end{equation}
whence
\begin{equation}\label{eq:lessthanone}
\big| \big\langle \bar{f}_{n,(i_1,j_1)}, \bar{f}_{n,(i_2,j_2)} \big\rangle_\Hil\big| \leqslant 1.
\end{equation}
We also write
\begin{equation*}
Z^{(n)}_{(s,t)} \eqdefl \varepsilon_n c^{-p/2} \big(V^{(p)}_{(s,t)}(k_n,n) - \E_\wn\big[V^{(p)}_{(s.t)}(k_n,n)\big] \big), \quad (s,t) \in [0,1]^2.
\end{equation*}

As a preparation for the proof of Theorem \ref{thm:powerclt}, we prove some key lemmas.
First, we obtain a uniform estimate for the decay of correlations \eqref{eq:incrcorr} under Assumption \ref{asm:clt}.

\begin{lem}[Correlation estimate]\label{lem:corest} If Assumption \ref{asm:clt} holds, then
\begin{equation*}
\overline{\rho}_n \eqdefl \sup_{(i_1,j_1)\neq(i_2,j_2)} \big|\big\langle \bar{f}_{n,(i_1,j_1)}, \bar{f}_{n,(i_2,j_2)} \big\rangle_\Hil\big| = o(\varepsilon_n).
\end{equation*}
\end{lem}

\begin{proof}
For any $(i_1,j_1)\neq(i_2,j_2)$, we have the bound
\begin{multline}\label{eq:sqcorr}
\big\langle \bar{f}_{n,(i_1,j_1)}, \bar{f}_{n,(i_2,j_2)} \big\rangle_\Hil^2
\\ \lesssim_\sigma \bigg(\int \dot{\pi}_n(\xi,\tau)^{1/2}\dot{\pi}_n\big(\xi+\varepsilon_n(i_1-i_2),\tau+\varepsilon_n(j_1-j_2)\big)^{1/2}\ud \xi \ud \tau\bigg)^2.
\end{multline}
Since $E_n \cap (E_n + (\varepsilon_n i,\varepsilon_n j))$ is a $\lambda_2$-null set for any $(i,j) \in \Z^2\setminus \{ \bs{0} \}$, we have
\begin{equation}\label{eq:integralsplit}
\int\limits_{\R^2} f(\bs{z}) \lambda_2(\ud \bs{z}) \leqslant \int\limits_{\R^2 \setminus E_n}f(\bs{z}) \lambda_2(\ud \bs{z})+ \int\limits_{\R^2 \setminus (E_n+ (\varepsilon_n(i_2-i_1),\varepsilon_n (j_2-j_1)))}f(\bs{z}) \lambda_2(\ud \bs{z}).
\end{equation}
for any $f \in L^1(\R^2)$.
By applying \eqref{eq:integralsplit}, using the inequality $(s+t)^2 \leqslant 2(s^2 + t^2)$,~$(s,t)\in \R^2$, the Cauchy--Schwarz inequality, and making the obvious change of variables, we find that the right-hand side of \eqref{eq:sqcorr} is bounded by \begin{equation*}4\pi_n \big(\R^2 \setminus E_n\big) = o(\varepsilon^2_n),
\end{equation*}
which is independent of $(i_1,j_1)$ and $(i_2,j_2)$.
\end{proof}

\revised{
\begin{rem}
As is evident from the proof of Lemma \ref{lem:corest}, the concentration measure $\pi_n$ provides a method to bound the correlations between the increments $Y\big(R^{(n)}_{(k_n i,k_n j)} \big)$, $i$,\ $j = 1,\ldots,\lfloor \varepsilon^{-1}_n \rfloor$. In general, these correlations seem to be difficult to evaluate or estimate precisely, unless the weight function $g$ factorizes as $g(s,t) = g_1(s)g_2(t)$ with some $g_1$,\, $g_2 \in L^2(\R)$, whereas the asymptotic behavior of the concentration measure $\pi_n$ as $n \rightarrow \infty$ is considerably more tractable even without such factorization of $g$, as we shall see in Section \ref{sec:weight}. However, the present concentration measure approach has the limitation that the uniform bound of Lemma \ref{lem:corest} might not be sharp, especially with increments over rectangles that are far apart.
\end{rem}
}

Next, we derive a chaos decomposition for the random field $Z^{(n)}$.
\begin{lem}[Chaos decomposition]
For any $n \in \N$ and $(s,t) \in [0,1]^2$,
\begin{equation}\label{eq:powerchaos}
Z^{(n)}_{(s,t)} = \sum_{k=2}^\infty I_k \big(F^{(n,k)}_{(s,t)}\big)\quad \textrm{in $L^2(\Omega_\wn)$,}
\end{equation}
where
\begin{equation*}
F^{(n,k)}_{(s,t)} \eqdefl \frac{\alpha_k}{k!} \varepsilon_n\sum_{i=1}^{\lfloor s/\varepsilon_n \rfloor}\sum_{j=1}^{\lfloor t/\varepsilon_n \rfloor} \bigg(\frac{\|f_{n,(i,j)}\|^2_\Hil}{c_n}\bigg)^{p/2}  \bar{f}^{\otimes k}_{n,(i,j)} \in \Hil^{\otimes k}.
\end{equation*}
\end{lem}

\begin{proof} Since \revised{$\|f_{n,(i,j)}\|_\Hil^{-1}Y\big(R^{(n)}_{(k_n i,k_n j)} \big) \sim N(0,1)$ } given $\sigma$, and since the Hermite rank of the function $u_p$ is $2$, we have the expansion
\begin{equation}\label{eq:hermitechaos}
\begin{split}
Z^{(n)}_{(s,t)} & = \varepsilon_n c^{-p/2}_n \sum_{i=1}^{\lfloor s/\varepsilon_n \rfloor}\sum_{j=1}^{\lfloor t/\varepsilon_n \rfloor}\big(\big|Y\big(R^{(n)}_{(k_n i,k_n j)} \big)\big|^p - \E_\wn\big[\big|Y\big(R^{(n)}_{(k_n i,k_n j)} \big)\big|^p\big] \big) \\
& = \varepsilon_n c^{-p/2}_n \sum_{i=1}^{\lfloor s/\varepsilon_n \rfloor}\sum_{j=1}^{\lfloor t/\varepsilon_n \rfloor} \|f_{n,(i,j)}\|_\Hil^{p} u_p \Big(\|f_{n,(i,j)}\|_\Hil^{-1}Y\big(R^{(n)}_{(k_n i,k_n j)} \big) \Big) \\
& = \varepsilon_n c^{-p/2}_n \sum_{i=1}^{\lfloor s/\varepsilon_n \rfloor}\sum_{j=1}^{\lfloor t/\varepsilon_n \rfloor} \|f_{n,(i,j)}\|_\Hil^{p} \sum_{k=2}^\infty \alpha_k H_k\Big(\|f_{n,(i,j)}\|_\Hil^{-1}Y\big(R^{(n)}_{(k_n i,k_n j)} \big) \Big).
\end{split}
\end{equation}
As
$\|f_{n,(i,j)}\|_\Hil^{-1}Y\big(R^{(n)}_{(k_n i,k_n j)} \big) = I_1\big(\bar{f}_{n,(i,j)}\big)$ \revised{and as $\|\bar{f}_{n,(i,j)}\|_\Hil=1$}, the Hermite representation of iterated Wiener integrals \cite[Theorem 13.25]{Kal2002} yields
\begin{equation}\label{eq:hermiteiter}
H_k\Big(\|f_{n,(i,j)}\|_\Hil^{-1}Y\big(R^{(n)}_{(k_n i,k_n j)} \big)\Big) = \frac{1}{k!}I_k\big( \bar{f}^{\otimes k}_{n,(i,j)}\big).
\end{equation}
By plugging \eqref{eq:hermiteiter} into \eqref{eq:hermitechaos} and rearranging, we arrive at the asserted chaos decomposition.
\end{proof}

\begin{rem}\label{rem:hermite}
Since $\alpha_2,\alpha_3,\ldots$ are the non-zero coefficients in the Hermite expansion of $u_p$, we have
\begin{equation}\label{eq:hermitecoef}
\sum_{k=2}^{\infty} \frac{\alpha^2_k}{k!} = \int u_p(x)^2 \gamma(\ud x) = m_{2p}-m^2_p <\infty.
\end{equation}
\end{rem}

We will use Lemma \ref{lem:chaosclt} to prove the convergence of the finite-dimensional distributions of $Z^{(n)}$, using the chaos decomposition \eqref{eq:powerchaos}. To this end, we study the asymptotic behavior of the kernels in \eqref{eq:powerchaos}.

\begin{lem}[Asymptotics of kernels]\label{lem:kerasy}
If Assumption \ref{asm:clt} holds, then for any $(s_1,t_1)$,\ $(s_2,t_2) \in [0,1]^2$, $k \geqslant 2$, and $r = 1,\ldots,k-1$,
\begin{align}
\label{eq:chnorm} &\lim_{m \rightarrow \infty} \limsup_{n\rightarrow\infty}\sum_{k=m}^{\infty} k!\big\|F^{(n,k)}_{(s_1,t_1)}\big\|^2_{\Hil^{\otimes k}} = 0,\\
\label{eq:chinner} &\lim_{n \rightarrow \infty} k!\big\langle F^{(n,k)}_{(s_1,t_1)}, F^{(n,k)}_{(s_2,t_2)} \big\rangle_{\Hil^{\otimes k}} = \frac{\alpha^2_k}{k!}\int\limits_{-s_0}^{s_1\wedge s_2 - s_0}\int\limits_{-t_0}^{t_1\wedge t_2 - t_0} \sigma_{(u,v)}^{2p}\ud u \ud v,\\
\label{eq:chcontr} &\lim_{n \rightarrow \infty} \big\| F^{(n,k)}_{(s_1,t_1)} \otimes_r F^{(n,k)}_{(s_1,t_1)}\big\|^2_{\Hil^{\otimes 2(k-r)}} = 0.
\end{align}
\end{lem}

\begin{proof}
Below, we use the index sets
\begin{equation*}\begin{split}
\mathcal{I}_n & \eqdefl \{ i_1 : 1 \leqslant i_1 \leqslant \lfloor s_1/\varepsilon_n \rfloor\} \times \{ j_1 : 1 \leqslant j_1 \leqslant \lfloor t_1/\varepsilon_n \rfloor\} \times \{ i_2 : 1 \leqslant i_2 \leqslant \lfloor s_2/\varepsilon_n \rfloor\} \\
& \qquad \times \{ j_2 : 1 \leqslant j_2 \leqslant \lfloor t_2/\varepsilon_n \rfloor\}
\end{split}
\end{equation*}
and
\begin{equation*}
\mathring{\mathcal{I}}_n \eqdefl \mathcal{I}_n \setminus \{ (i,j,i,j) : 1 \leqslant i \leqslant \lfloor (s_1\wedge s_2)/\varepsilon_n \rfloor, \, 1 \leqslant j \leqslant \lfloor (t_1 \wedge t_2)/\varepsilon_n\rfloor \}.
\end{equation*}
Let us expand
\begin{multline*}
k!\big\langle F^{(n,k)}_{(s_1,t_2)}, F^{(n,k)}_{(s_2,t_2)} \big\rangle_{\Hil^{\otimes k}} \\ 
= \frac{\alpha^2_k}{k!} \varepsilon^2_n \sum_{(i_1,j_1,i_2,j_2) \in \mathcal{I}_n} \bigg(\frac{\|f_{n,(i_1,j_1)}\|^2_\Hil \|f_{n,(i_2,j_2)}\|^2_\Hil}{c^2_n}\bigg)^{p/2}    \big\langle \bar{f}^{\otimes k}_{n,(i_1,j_1)}, \bar{f}^{\otimes k}_{n,(i_2,j_2)} \big\rangle_{\Hil^{\otimes k}},
\end{multline*}
where $\big\langle \bar{f}^{\otimes k}_{n,(i_1,j_1)}, \bar{f}^{\otimes k}_{n,(i_2,j_2)} \big\rangle_{\Hil^{\otimes k}} = \big\langle \bar{f}_{n,(i_1,j_1)}, \bar{f}_{n,(i_2,j_2)} \big\rangle^k_{\Hil}$.  As $k \geqslant 2$, we have by \eqref{eq:lessthanone},
\begin{equation*}
\begin{split}
\sup_{(i_1,j_1,i_2,j_2) \in \mathring{\mathcal{I}}_n}\big|\big\langle \bar{f}_{n,(i_1,j_1)}, \bar{f}_{n,(i_2,j_2)} \big\rangle_{\Hil}^k\big| & = \sup_{(i_1,j_1,i_2,j_2) \in \mathring{\mathcal{I}}_n} \big|\big\langle \bar{f}_{n,(i_1,j_1)}, \bar{f}_{n,(i_2,j_2)} \big\rangle_{\Hil}\big|^k\\
& \leqslant \sup_{(i_1,j_1,i_2,j_2) \in \mathring{\mathcal{I}}_n}  \big|\big\langle \bar{f}_{n,(i_1,j_1)}, \bar{f}_{n,(i_2,j_2)} \big\rangle_{\Hil}\big|^2 \\
& \leqslant \overline{\rho}^2_n =  o(\varepsilon^2_n).
\end{split}
\end{equation*}
\revised{The boundedness of $\sigma$ implies that $\|f_{n,(i,j)}\|^2_{\Hil}\lesssim_\sigma c_n$.
Thus, we can write
\begin{equation*}
k!\big\langle F^{(n,k)}_{(s_1,t_2)}, F^{(n,k)}_{(s_2,t_2)} \big\rangle_{\Hil^{\otimes k}} = \frac{\alpha^2_k}{k!} \Bigg(\varepsilon^2_n\sum_{i=1}^{\lfloor (s_1 \wedge s_2)/\varepsilon_n\rfloor} \sum_{j=1}^{\lfloor (t_1 \wedge t_2)/\varepsilon_n\rfloor} \bigg(\frac{\|f_{n,(i,j)}\|^2_\Hil}{c_n}\bigg)^{p} +  \Theta_{k,n} \Bigg),
\end{equation*}
where 
\begin{equation*}
\sup_{k \geqslant 2}|\Theta_{k,n}| \lesssim_{\sigma,p} \varepsilon^2_n \overline{\rho}^2_n|\mathring{\mathcal{I}}_n| \xrightarrow[n \rightarrow \infty]{} 0
\end{equation*}
by the bound $|\mathring{\mathcal{I}}_n| \lesssim_{(s_1,t_1),(s_2,t_2)} \varepsilon^{-4}_n$. }
 Now \eqref{eq:chinner} follows since
\begin{multline*}
\varepsilon^2_n\sum_{i=1}^{\lfloor(s_1 \wedge s_2)/\varepsilon_n\rfloor} \sum_{j=1}^{\lfloor (t_1 \wedge t_2)/\varepsilon_n\rfloor} \bigg(\frac{\|f_{n,(i,j)}\|^2_\Hil}{c_n}\bigg)^{p}\\  = \varepsilon^2_n\sum_{i=1}^{\lfloor (s_1 \wedge s_2)/\varepsilon_n\rfloor} \sum_{j=1}^{\lfloor (t_1 \wedge t_2)/\varepsilon_n\rfloor} \bigg( \int \sigma_{(\varepsilon_n i-\xi, \varepsilon_n j-\tau)}^2\pi_n(\ud \xi, \ud \tau)\bigg)^p \\
 \xrightarrow[ n\rightarrow \infty]{} \int_0^{s_1\wedge s_2} \int_0^{t_1\wedge t_2} \sigma_{(u,v)}^{2p}\ud u \ud v,
\end{multline*}
(cf.~the proof of Theorem \ref{thm:powerlln}).
Moreover, by the reverse Fatou's lemma and \eqref{eq:hermitecoef}, we obtain
\begin{equation*}
\revised{\limsup_{n\rightarrow\infty}\sum_{k=m}^{\infty} k!\big\|F^{(n,k)}_{(s_1,t_1)}\big\|^{2}_{\Hil^{\otimes k}}} \leqslant \sum_{k=m}^\infty \frac{\alpha^2_k}{k!} \int_0^{s_1} \int_0^{t_1} \sigma_{(u,v)}^{2p}\ud u \ud v \xrightarrow[ m\rightarrow \infty]{} 0,
\end{equation*}
establishing \eqref{eq:chnorm}.

To show \eqref{eq:chcontr}, we may use the bound
\begin{multline*}
\big\| F^{(n,k)}_{(s_1,t_1)}\otimes_r F^{(n,k)}_{(s_1,t_1)}\big\|^2_{\Hil^{\otimes 2(k-r)}} \\ 
\begin{aligned}
& \lesssim_{\sigma,k}  \varepsilon^4_n \sum_{(\mathbf{i}_1,\mathbf{i}_2,\mathbf{i}_3,\mathbf{i}_4) \in\mathcal{J}_n} \big\langle \bar{f}^{\otimes k}_{n,\mathbf{i}_1} \otimes_r \bar{f}^{\otimes k}_{n,\mathbf{i}_2}, \bar{f}^{\otimes k}_{n,\mathbf{i}_3} \otimes_r \bar{f}^{\otimes k}_{n,\mathbf{i}_4}\big\rangle_{\Hil^{\otimes 2(k-r)}} \\
& = \varepsilon^4_n\sum_{(\mathbf{i}_1,\mathbf{i}_2,\mathbf{i}_3,\mathbf{i}_4) \in\mathcal{J}_n}  \big\langle \bar{f}_{n,\mathbf{i}_1}, \bar{f}_{n,\mathbf{i}_2}\big\rangle^r_\Hil \big\langle \bar{f}_{n,\mathbf{i}_3}, \bar{f}_{n,\mathbf{i}_4}\big\rangle^r_\Hil   \big\langle \bar{f}_{n,\mathbf{i}_1}, \bar{f}_{n,\mathbf{i}_3}\big\rangle^{k-r}_\Hil \big\langle \bar{f}_{n,\mathbf{i}_2}, \bar{f}_{n,\mathbf{i}_4}\big\rangle^{k-r}_\Hil,
\end{aligned}
\end{multline*}
where
\begin{equation*}
\mathcal{J}_n \eqdefl (\{1,\ldots,\lfloor s_1/\varepsilon_n \rfloor \} \times \{1,\ldots,\lfloor t_1/\varepsilon_n \rfloor \})^4.
\end{equation*}
Since $r \geqslant 1$ and $k-r\geqslant 1$, we have by \eqref{eq:lessthanone},
\begin{multline}\label{eq:fourinnerproducts}
\varepsilon^4_n\sum_{(\mathbf{i}_1,\mathbf{i}_2,\mathbf{i}_3,\mathbf{i}_4) \in\mathcal{J}_n}  \big\langle \bar{f}_{n,\mathbf{i}_1}, \bar{f}_{n,\mathbf{i}_2}\big\rangle^r_\Hil \big\langle \bar{f}_{n,\mathbf{i}_3}, \bar{f}_{n,\mathbf{i}_4}\big\rangle^r_\Hil   \big\langle \bar{f}_{n,\mathbf{i}_1}, \bar{f}_{n,\mathbf{i}_3}\big\rangle^{k-r}_\Hil \big\langle \bar{f}_{n,\mathbf{i}_2}, \bar{f}_{n,\mathbf{i}_4}\big\rangle^{k-r}_\Hil\\
\leqslant \varepsilon^4_n \sum_{(\mathbf{i}_1,\mathbf{i}_2,\mathbf{i}_3,\mathbf{i}_4) \in\mathcal{J}_n} \big| \big\langle \bar{f}_{n,\mathbf{i}_1}, \bar{f}_{n,\mathbf{i}_2}\big\rangle_\Hil\big| \big| \big\langle \bar{f}_{n,\mathbf{i}_3}, \bar{f}_{n,\mathbf{i}_4}\big\rangle_\Hil\big| \big| \big\langle \bar{f}_{n,\mathbf{i}_1}, \bar{f}_{n,\mathbf{i}_3}\big\rangle_\Hil\big| \big| \big\langle \bar{f}_{n,\mathbf{i}_2}, \bar{f}_{n,\mathbf{i}_4}\big\rangle_\Hil\big|
\end{multline}
We use a simple combinatorial argument to deduce that the right-hand side (r.h.s.)\ of  the inequality \eqref{eq:fourinnerproducts} tends to zero. To this end, we define $u_n : \mathcal{J}_n \longrightarrow \{0,1,2,4\}$ by
\begin{equation*}
u_n(\mathbf{i}_1,\mathbf{i}_2,\mathbf{i}_3,\mathbf{i}_4) = \mathbf{1}_{\{\mathbf{i}_1=\mathbf{i}_2 \}} + \mathbf{1}_{\{\mathbf{i}_3=\mathbf{i}_4 \}} + \mathbf{1}_{\{\mathbf{i}_1=\mathbf{i}_3 \}} + \mathbf{1}_{\{\mathbf{i}_2=\mathbf{i}_4 \}},
\end{equation*}
 where the value $3$ is, indeed, never attained. It is straightforward to check that
\begin{align*}
|u_n^{-1}(\{ 0\})| & \lesssim_{(s_1,t_1)} \varepsilon^{-8}_n, & |u_n^{-1}(\{ 1\})| & \lesssim_{(s_1,t_1)} \varepsilon^{-6}_n, \\
|u_n^{-1}(\{ 2\})| & \lesssim_{(s_1,t_1)} \varepsilon^{-4}_n, & |u_n^{-1}(\{ 4\})| & \lesssim_{(s_1,t_1)} \varepsilon^{-2}_n,
\end{align*}
for all $n \in \N$.
Hence, splitting the summation on the r.h.s.\ of \eqref{eq:fourinnerproducts} by
\begin{equation*}
\sum_{\mathcal{J}_n} = \sum_{u_n^{-1}(\{ 0\})} + \sum_{u_n^{-1}(\{ 1\})} + \sum_{u_n^{-1}(\{ 2\})} + \sum_{u_n^{-1}(\{ 4\})}
\end{equation*}
and using Lemma \ref{lem:corest} leads to the bound
\begin{equation*}
\textrm{r.h.s.\ of \eqref{eq:fourinnerproducts}} \lesssim_{\sigma,k,(s_1,t_1)} \varepsilon_n^4 \big(\varepsilon^{-8}_n o(\varepsilon^4_n) + \varepsilon^{-6}_n o(\varepsilon^3_n) + \varepsilon^{-4}_n o(\varepsilon^2_n) + \varepsilon^{-2}_n\big) = o(1). \mbox{\qedhere}
\end{equation*}
\end{proof}

We are now ready to proceed to the actual proof of Theorem \ref{thm:powerclt}, building on the preceding three lemmas.

\begin{proof}[Proof of Theorem \ref{thm:powerclt}]

In this proof, unlike in the rest of the paper, the space $D([0,1]^2)$ is endowed with the \emph{Skorohod topology} (see Appendix \ref{app:skorohodspace}). \revised {Ultimately, we can switch to the uniform topology by Lemma \ref{lem:skorohod}, as $\Xi^{(p)}$ is a continuous random field. } 

\emph{Step 1:~Reductions.} 
It is clearly sufficient to show that
\begin{equation*}
Z^{(n)} \xrightarrow[n\rightarrow \infty]{L_{\mathcal{F}_\wn \otimes \mathcal{F}_\sigma}} \overline{\Xi}^{(p)} \quad \textrm{in $D([0,1]^2)$,}
\end{equation*}
where $\overline{\Xi}^{(p)} = (m_{2p}-m^2_p)^{1/2}\Xi^{(p)}$. The quadrant Brownian sheets $\big(W_{(x,t)}^{(i)}\big)_{(x,t) \in [0,1]^2}$, for $i= 1,2,3,4$, defined by
\begin{align*}  
W^{(1)}_{(s,t)} & \eqdefl I_1\big(\mathbf{1}_{[0,s]\times[0,t]}\big), & 
W^{(2)}_{(s,t)} & \eqdefl I_1\big(\mathbf{1}_{[-s,0]\times[0,t]}\big), \\
W^{(3)}_{(s,t)} & \eqdefl I_1\big(\mathbf{1}_{[-s,0]\times[-t,0]}\big), & 
W^{(4)}_{(s,t)} & \eqdefl I_1\big(\mathbf{1}_{[0,s]\times[-t,0]}\big),
\end{align*}
(modulo taking continuous modifications) generate the $\sigma$-algebra $\mathcal{F}_\wn$. 
Thus, by Lemma \ref{lem:stable}, it is sufficient to show that for any continuous, bounded test function $\varphi :  C([0,1]^2)^4 \times C([-1,1]^2)\times D([0,1]^2) \longrightarrow \R$,
\begin{equation}\label{eq:testfunction}
\E\big[\varphi\big(W^{(1)},W^{(2)},W^{(3)},W^{(4)},\sigma,Z^{(n)}\big)\big] 
\xrightarrow[n\rightarrow \infty]{}\E\big[\varphi\big(W^{(1)},W^{(2)},W^{(3)},W^{(4)},\sigma,\overline{\Xi}^{(p)}\big)\big],
\end{equation}
where $\sigma$ is stochastic. But, in the view of Fubini's theorem, it is clear that \eqref{eq:testfunction} follows if we simply show that
\begin{equation}\label{eq:sigmafixed}
\big(W^{(1)},W^{(2)},W^{(3)},W^{(4)},Z^{(n)}\big) \xrightarrow[n\rightarrow \infty]{L}\big(W^{(1)},W^{(2)},W^{(3)},W^{(4)},\overline{\Xi}^{(p)}\big)
\end{equation}
in $C([0,1]^2)^4\times D([0,1]^2)$, with the realization of $\sigma$ kept fixed. As usual, we will prove \eqref{eq:sigmafixed} by establishing convergence of the finite-dimensional distributions first, and then showing tightness.

\emph{Step 2:~Convergence of finite-dimensional distributions.} We fix arbitrary $d \in \N$ and $(\bs{s},\bs{t})\eqdefl \big((s_1,t_1),\ldots,(s_d,t_d)\big) \in [0,1]^{2d}$. Let us denote for any $i = 1,2,3,4$,
\begin{equation*}
W^{(i)}_{(\bs{s},\bs{t})}\eqdefl\big(W^{(i)}_{(s_1,t_1)},\ldots,W^{(i)}_{(s_d,t_d)}\big)
\end{equation*}
and for any $n \in \N$,
\begin{equation*}
Z^{(n)}_{(\bs{s},\bs{t})}\eqdefl\big(Z^{(n)}_{(s_1,t_1)},\ldots,Z^{(n)}_{(s_d,t_d)}\big).
\end{equation*}
We would like to show that
\begin{equation}\label{eq:fdd}
\big(W^{(1)}_{(\bs{s},\bs{t})}, W^{(2)}_{(\bs{s},\bs{t})}, W^{(3)}_{(\bs{s},\bs{t})}, W^{(4)}_{(\bs{s},\bs{t})}, Z^{(n)}_{(\bs{s},\bs{t})}\big) \xrightarrow[n\rightarrow \infty]{L} N_{5d}\bigg(\bs{0}, \, \begin{bmatrix} \bs{\Psi}^{(1)} & \bs{0} \\ \bs{0} & \bs{\Psi}^{(2)}\end{bmatrix}\bigg),
\end{equation}
where $\bs{\Psi}^{(1)}$ is the covariance matrix of $\big(W^{(1)}_{(\bs{s},\bs{t})}, W^{(2)}_{(\bs{s},\bs{t})}, W^{(3)}_{(\bs{s},\bs{t})}, W^{(4)}_{(\bs{s},\bs{t})} \big)$ and
\begin{equation*}
\bs{\Psi}^{(2)}_{i,j} \eqdefl (m_{2p}-m^2_p) \int\limits_{-s_0}^{s_i\wedge s_j-s_0}\int\limits_{-t_0}^{t_i\wedge t_j-t_0} \sigma_{(u,v)}^{2p}\ud u \ud v, \quad i,\, j = 1,\ldots,d.
\end{equation*} 
Note that $W^{(1)}_{(\bs{s},\bs{t})}$, $W^{(2)}_{(\bs{s},\bs{t})}$, $W^{(3)}_{(\bs{s},\bs{t})}$, and $W^{(4)}_{(\bs{s},\bs{t})}$ are vectors of first-order Wiener integrals, whereas the chaos decompositions of the components of $Z^{(n)}_{(\bs{s},\bs{t})}$ do not have any contributions from first-order integrals. Thus, \eqref{eq:fdd} holds whenever the chaos decompositions of $Z^{(n)}_{(\bs{s},\bs{t})}$,~$n\in \N$ satisfy the conditions of Lemma \ref{lem:chaosclt}, which is indeed the case due to Remark \ref{rem:hermite} and Lemma \ref{lem:kerasy}. 

\emph{Step 3:~Tightness.} Since marginal tightness implies joint tightness, it suffices to show that the sequence $\big(Z^{(n)}\big)_{n \in \N}$ is tight in $D([0,1]^2)$. Let us write
\begin{equation*}
\mathcal{T}_n \eqdefl \{ \varepsilon_n i : i = 0,1,\ldots,\lfloor \varepsilon_n^{-1} \rfloor \} \cup \{ 1 \}, \quad n \in \N.
\end{equation*}
Moreover, let $R$ be a rectangle with vertices in $\mathcal{T}^2_n$, that is, $R = (s_1,s_2]\times (t_1,t_2]$ for some $s_1,s_2,t_1,t_2\in \mathcal{T}_n$ such that $s_1 < s_2$ and $t_1 < t_2$.
By \cite[pp.\ 1658,\ 1665]{BW1971}, it is sufficient to show that 
\begin{equation}\label{eqn:tightnessbound}
\E_\wn\big[Z^{(n)}(R)^4\big] \lesssim_{\sigma,p} \lambda_2(R)^2.
\end{equation}
Recall that
\begin{equation*}
\begin{split}
Z^{(n)}(R) & = Z^{(n)}_{(s_2,t_2)}-Z^{(n)}_{(s_2,t_1)}-Z^{(n)}_{(s_1,t_2)}+Z^{(n)}_{(s_1,t_1)} \\
& = \varepsilon_n c^{-p/2}_n \sum_{i=\lfloor s_1/\varepsilon_n \rfloor+1}^{\lfloor s_2/\varepsilon_n \rfloor}\sum_{j=\lfloor t_1/\varepsilon_n \rfloor+1}^{\lfloor t_2/\varepsilon_n \rfloor} \|f_{n,(i,j)}\|_\Hil^{p} u_p \Big(\|f_{n,(i,j)}\|_\Hil^{-1}Y\big(R^{(n)}_{(k_n i,k_n j)} \big) \Big).
\end{split}
\end{equation*}
Since $\overline{\rho}_n \rightarrow 0$, there exists $n_0 \in \N$ such that $\sup_{n \geqslant n_0} \overline{\rho}_n < 1/12$. Thus, by Lemma \ref{lem:4thmom}, below, we have for all $n \geqslant n_0$,
\begin{equation*}
\begin{split} 
\E_\wn\big[Z^{(n)}(R)^4\big] & \lesssim_\sigma \varepsilon^4_n \E_\wn \bigg[ \bigg( \sum_{i=\lfloor s_1/\varepsilon_n \rfloor+1}^{\lfloor s_2/\varepsilon_n \rfloor}\sum_{j=\lfloor t_1/\varepsilon_n \rfloor+1}^{\lfloor t_2/\varepsilon_n \rfloor} u_p \Big(\|f_{n,i,j}\|_\Hil^{-1}Y\big(R^{(n)}_{(k_n i,k_n j)} \big) \Big)\bigg)^4\bigg] \\
& \lesssim_p \varepsilon^4_n \big(\varepsilon_n^{-8} \lambda_2(R)^4 \overline{\rho}_n^4 + \varepsilon_n^{-6} \lambda_2(R)^3 \overline{\rho}^2_n + \varepsilon_n^{-4} \lambda_2(R)^2\big) \\
& \leqslant\lambda_2(R)^2 \big(\varepsilon^{-4}_n \overline{\rho}^4_n+\varepsilon_n^{-2}\overline{\rho}^2_n + 1\big),
\end{split}
\end{equation*}
where we used \eqref{eq:cequiv} and the inequalities
$(\lfloor s_2/\varepsilon_n \rfloor-\lfloor s_1/\varepsilon_n \rfloor) (\lfloor t_2/\varepsilon_n \rfloor - \lfloor t_1/\varepsilon_n \rfloor) \leqslant \varepsilon_n^{-2} \lambda_2(R)$ and $\lambda_2(R) \leqslant 1$. Finally, $\overline{\rho}_n = o(\varepsilon_n)$ implies that
\begin{equation*}
\sup_{n \in \N} \big(\varepsilon^{-4}_n \overline{\rho}^4_n+\varepsilon_n^{-2}\overline{\rho}^2_n + 1\big) < \infty,
\end{equation*}
whence \eqref{eqn:tightnessbound} holds.
\end{proof}

It remains to prove the moment estimate stated in Lemma \ref{lem:4thmom} below, which we used above to establish tightness. It is similar to --- albeit much less general than --- Proposition 4.2 of \cite{Taqqu1977}. The key difference, however, is that unlike in \cite{Taqqu1977}, here the underlying Gaussian random variables do not form a \emph{stationary} process. (Alas, the assumption of stationarity renders Proposition 4.2 of \cite{Taqqu1977} unapplicable in our setting.) The proof relies on a product moment bound due to Soulier \cite[Corollary 2.1]{Sou2001} and a simple combinatorial argument.

\begin{lem}[Fourth moment]\label{lem:4thmom}
Let $(X_1,\ldots,X_n)$ be a Gaussian random vector such that $\E[X_i]=0$ and $\E[X^2_i]=1$ for all $i = 1,\ldots,n$. If $f\in L^2(\R,\gamma)$ has Hermite rank $r \in \N$ and
$
\rho \eqdefl \sup_{i\neq j} |\E[X_iX_j]| \leqslant \rho^*$ for some $\rho^* \in (0,1/12)$,
then
\begin{equation*}
\bigg|\E\bigg[\bigg(\sum_{i=1}^n f(X_i)\bigg)^4 \bigg]\bigg| \lesssim_{f,\rho^*} n^{4} \rho^{2r} + n^3 \rho^r + n^2.
\end{equation*}
\end{lem}

\begin{proof}
We use first the trivial bound
\begin{equation}\label{eq:fourpoint}
\bigg|\E\bigg[\bigg(\sum_{i=1}^n f(X_i)\bigg)^4 \bigg]\bigg| \leqslant \sum_{i_1,i_2,i_3,i_4 = 1}^n |\E[f(X_{i_1})f(X_{i_2})f(X_{i_3})f(X_{i_4})]|.
\end{equation}
By Corollary 2.1 of \cite{Sou2001}, we have 
\begin{equation}\label{eq:4corrbound}
|\E[f(X_{i_1})f(X_{i_2})f(X_{i_3})f(X_{i_4})]| \lesssim_{f,\rho^*} \rho^{r\tilde{u}_n(i_1,i_2,i_3,i_4)/2},
\end{equation}
where
\begin{equation*}
\tilde{u}_n (i_1,i_2,i_3,i_4) \eqdefl |\{i : \textrm{$i = i_k$ for some $k$ and $i \neq i_l$ for any $l \neq k$}  \}|
\end{equation*}
is the number of unrepeated indices in $(i_1,i_2,i_3,i_4)$ (cf.~the function $u_n$ in the proof of Lemma \ref{lem:kerasy}).
 Note that $\tilde{u}_n$ is a mapping from $\{1,\ldots,n \}^4$ onto $\{0,1,2,4 \}$, since it is impossible to have exactly three indices that are not repeated. It is key to observe that
\begin{align*}
|\tilde{u}_n^{-1}(\{0\})| & \lesssim n^2 + n \lesssim n^2, & |\tilde{u}_n^{-1}(\{1\})| & \lesssim n^2, \\
|\tilde{u}_n^{-1}(\{2\})| & \lesssim n^3, & |\tilde{u}_n^{-1}(\{4\})| & \leqslant n^4,
\end{align*}
for all $n \in\N$. (In the case $\tilde{u}_n(i_1,i_2,i_3,i_4)=0$, either $i_1=i_2=i_3=i_4$ or there are two distinct pairs of repeated indices.) Thus, by \eqref{eq:fourpoint} and \eqref{eq:4corrbound}, we obtain
\[
\bigg|\E\bigg[\bigg(\sum_{i=1}^n f(X_i)\bigg)^4 \bigg]\bigg| \lesssim_{f,\rho^*} n^4 \rho^{2r} + n^3 \rho^r + n^2 \rho^{r/2} + n^2 \leqslant n^{4} \rho^{2r} + n^3 \rho^r + n^2,
\]
which completes the proof.
\end{proof}

\section{Asymptotics of concentration measures}\label{sec:weight}

In this section, we prove Proposition \ref{prop:examplekernel} by deriving polynomial estimates for the integrals of the weight function $g$ over some certain decisive subsets of $\R^2$. As the proof of Proposition \ref{prop:trianglekernel} uses a closely related argument, we merely sketch its main points.

\subsection{Proof of Proposition \ref{prop:examplekernel}}\label{sec:singular}

Let $g \in L^2(\R^2)$ be given by \eqref{eq:gspec}. It will be convenient to consider the non-normalized measures
\begin{equation*}
\mu_n(\ud s, \ud t) \eqdefl h_n(s,t)^2\ud s \ud t, \quad n \in \N.
\end{equation*}
Note that the support of $\mu_n$, like $\pi_n$, is contained in $[0,1+1/n]^2$. Under the present assumptions $g$ is a symmetric function, which clearly implies that also $h_n$ is symmetric. Thus, as $\mu_n$ is absolutely continuous with respect to the Lebesgue measure, we have
 \begin{equation}\label{eq:triangle}
 \pi_n \big(\R^2 \setminus E_n \big) = \frac{\mu_n\big( (0,1+1/n)^2\setminus (0,\varepsilon_n)^2\big)}{\mu_n\big( (0,1+1/n)^2\big)} = \frac{\mu_n\big( T_n\setminus (0,\varepsilon_n)^2\big)}{\mu_n( T_n )},
 \end{equation}
where $T_n \eqdefl \{ (s,t) : 0 < t < s < 1+ 1/n \}$. On certain subsets of $T_n$, the expression for $h_n(s,t)$ can be simplified significantly. To make use of this fact, we define for any $n \in \N$,
$\tilde{E}_n \eqdefl \{ (s,t) : 0<t<s <1/n\}$
and
\begin{align*}
B^{(1)}_n & \eqdefl (\varepsilon_n,1) \times (0,1/n),\\ B^{(2)}_n & \eqdefl \{(s,t) : \varepsilon_n < s < 1, \, s-1/n < t < s\}, \\
B^{(3)}_n & \eqdefl \{(s,t) : \varepsilon_n < s < 1 + 1/n, \, 1/n < t < (s-1/n)  \}, \\
B^{(4)}_n & \eqdefl (1,1+1/n)\times (0,1/n) \cup \{ (s,t) : 1 < s <1+1/n, \, s-1/n < t <s\}.
\end{align*}
(See Figure \ref{fig:partition}.)
\begin{figure}
  \myfigure{scale=0.9}{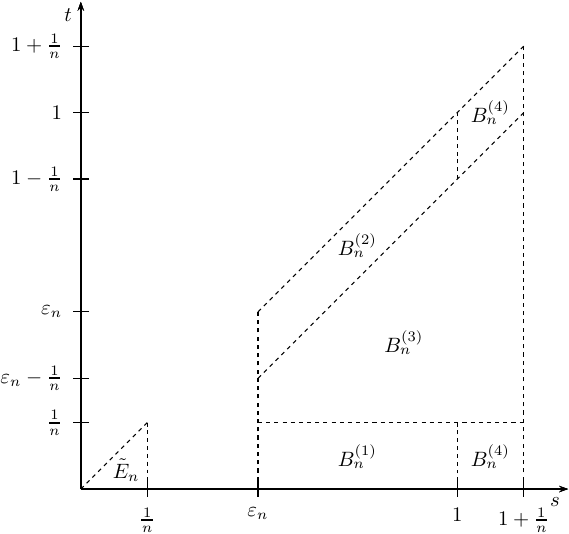}
  \caption{\label{fig:partition} The sets $\tilde{E}_n$, $B^{(1)}_n$, $B^{(2)}_n$, $B^{(3)}_n$, and $B^{(4)}_n$.}
\end{figure}
To prepare for the proof of Proposition \ref{prop:examplekernel}, we establish next polynomial estimates for the asymptotic behavior of the $\mu_n$-measures of some decisive subsets of $T_n$ as $n \rightarrow \infty$. 
\begin{lem}[Polynomial bounds]\label{lem:polybounds} Suppose that the assumptions of Proposition \ref{prop:examplekernel} hold and denote 
\begin{equation*}
n_0 \eqdefl \inf\{ n \geq 4 : \textrm{$k_n \geqslant 2$, $|\ell(x)|>0$ for all $x \in (0,1/n)$} \}. 
\end{equation*}
Then,
\begin{enumerate}[label=(\arabic*),ref=\arabic*]
\item\label{it:singularity} $\mu_n\big((0,\varepsilon_n)^2\cap T_n\big) \gtrsim_{\alpha,\ell} n^{-2(1-\alpha)}$ for all $n \geqslant n_0$,
\item\label{it:slow1} $\mu_n\big(B^{(1)}_n\big)=O(n^{-3+\kappa (2\alpha+1)})$,
\item\label{it:slow2} $\mu_n\big(B^{(2)}_n\big)=O(n^{-3+\kappa (2\alpha+1)})$,
\item\label{it:vanish} $\mu_n\big(B^{(3)}_n\big)= 0$ for all $n \geqslant n_0$,
\item\label{it:boundary} $\mu_n\big(B^{(4)}_n\big)=o(n^{-2})$.
\end{enumerate}
\end{lem}

\begin{proof}
Throughout the proof, we denote $f(s) \eqdefl s^{-\alpha}\ell(s)$ for all $s\in(0,1)$. It is straightforward to check that
\begin{equation}\label{eq:hsimplify}
h_n(s,t) = \begin{cases} f(s), & (s,t) \in \tilde{E}_n,\\
f(s) - f(s-1/n), & (s,t) \in B^{(1)}_n,\\
f(s-1/n) - f(t), & (s,t) \in B^{(2)}_n,\\
0, & (s,t) \in B^{(3)}_n.
\end{cases}
\end{equation}

\eqref{it:singularity}
The inclusion $\tilde{E}_n  \subset (0,\varepsilon_n)^2 \cap T_n$ implies that
$\mu_n\big(\tilde{E}_n\big) \leqslant \mu_n\big((0,\varepsilon_n)^2 \cap T_n\big)$.
By \eqref{eq:hsimplify}, we have
\begin{equation*}
\mu_n\big(\tilde{E}_n\big)  = \int_0^{1/n} \int_0^s \ud t f(s)^2 \ud s = \int_0^{1/n} sf(s)^2\ud s,
\end{equation*}
where $f(s)^2 \geqslant s^{-2\alpha} \inf_{u \in (0,1/n)} \ell(u)^2\gtrsim_{\ell} s^{-2\alpha}$.
Thus,
\begin{equation*}
\int_0^{1/n} sf(s)^2\ud s \gtrsim_\ell \int_0^{1/n} s^{-2\alpha + 1} \ud s = \frac{n^{-2(1-\alpha)}}{2(1-\alpha)}.
\end{equation*}

\eqref{it:slow1} Due to \eqref{eq:hsimplify}, we may write
\begin{equation}\label{eq:fdiff}
\begin{split}
\mu_n\big(B^{(1)}_n \big) & = \int_0^{1/n} \ud t \int_{\varepsilon_n}^{1} \big(f(s) - f(s-1/n)\big)^2\ud s\\
& = \frac{1}{n} \int^1_{\varepsilon_n} \big(f(s)-f(s-1/n)\big)^2 \ud s.
\end{split}
\end{equation}
By the mean value theorem, for any $s \in (\varepsilon_n,1)$, there exist $\xi_s \in [s-1/n,s]$ such that
\begin{equation}\label{eq:mvt}
f(s)-f(s-1/n) = \frac{1}{n}f'(\xi_s), 
\end{equation}
where
\begin{equation*}
f'(\xi_s) = -\alpha \xi_s^{-\alpha-1}\ell(\xi_s) + \xi_s^{-\alpha}\ell'(\xi_s).
\end{equation*}
Thus, we have the bounds
\begin{equation}\label{eq:diffbound}
\begin{split}
f'(\xi_s)^2 & \leqslant 2(\alpha^2 \xi_s^{-2(\alpha+1)}\| \ell \|^2_\infty + \xi_s^{-2\alpha} \|\ell'\|^2_\infty) \\
& \lesssim_{\alpha,\ell} (1+\xi_s^2)\xi_s^{-2(\alpha+1)}\\
& \leqslant 2(s-1/n)^{-2(\alpha+1)}.
\end{split}
\end{equation}
By plugging \eqref{eq:mvt} into \eqref{eq:fdiff} and applying \eqref{eq:diffbound} we arrive at
\begin{equation}\label{eq:finpowerbound}\begin{split}
\frac{1}{n} \int^1_{\varepsilon_n} \big(f(s)-f(s-1/n)\big)^2 \ud s &\lesssim_{\alpha,\ell} n^{-3}\int_{\varepsilon_n - 1/n}^\infty s^{-2(\alpha+1)}\ud s\\
& \lesssim_{\alpha} n^{-3} \varepsilon_n^{-2\alpha-1} = O(n^{-3+\kappa(2\alpha+1)}).
\end{split}
\end{equation}

\eqref{it:slow2} Proceeding as above, we have by \eqref{eq:hsimplify}, 
\begin{equation*}
\mu_n\big(B^{(2)}_n\big) = \int_{\varepsilon_n}^{1}\bigg(\int^s_{s-1/n}\big(f(s-1/n)-f(t)\big)^2 \ud t \bigg)\ud s,
\end{equation*}
where $0 \leqslant t-(s-1/n) \leqslant 1/n$. Thus,
\begin{equation*}
\big(f(s-1/n)-f(t)\big)^2 \lesssim_{\alpha,\ell} n^{-2}( s-1/n)^{-2(\alpha+1)},
\end{equation*}
by the mean value theorem and bounds analogous to \eqref{eq:diffbound}.
Moreover,
\begin{equation*}
\int_{\varepsilon_n}^{1}\bigg(\int^s_{s-1/n}\big(f(s-1/n)-f(t)\big)^2 \ud t \bigg)\ud s \lesssim_{\alpha,\ell} n^{-3}\int_{\varepsilon_n-1/n}^{\infty} s^{-2(\alpha+1)}\ud s,
\end{equation*}
and the assertion follows from \eqref{eq:finpowerbound}.

\eqref{it:vanish} 
Obvious, by \eqref{eq:hsimplify}.

\eqref{it:boundary} 
The estimate follows by observing that
\begin{equation*}
\mu_n \big(B^{(4)}_n\big) \leqslant \lambda_2\big(B^{(4)}_n\big) \sup_{(s,t)\in B^{(4)}_n} h_n(s,t)^2,
\end{equation*}
where $\lambda_2\big(B^{(4)}_n\big) \leqslant 2/n^2$ and $\sup_{(s,t)\in B^{(4)}_n} h_n(s,t)^2 \rightarrow 0$ as $n \rightarrow \infty$ because of the boundary condition $\lim_{s \rightarrow 1-}\ell(s)=0$.
\end{proof}

\begin{proof}[Proof of Proposition \ref{prop:examplekernel}]
As mentioned above, Assumption \ref{asm:clt} (in any form) implies Assumption \ref{asm:lln} with $\pi = \delta_{\bs{z}_0}$. Thus, in view of \eqref{eq:triangle}, it suffices to show that
\begin{equation*}
\varepsilon^{-2}_n\pi_n( \R^2 \setminus E_n) = \frac{\varepsilon^{-2}_n\mu_n\big( T_n \setminus (0,\varepsilon_n)^2\big)}{\mu_n\big( T_n \setminus (0,\varepsilon_n)^2\big)+ \mu_n\big( (0,\varepsilon_n)^2 \cap T_n \big)} \xrightarrow[n \rightarrow \infty]{}0,
\end{equation*}
which is equivalent to
\begin{equation}\label{eq:muconv}
M_n \eqdefl \frac{\varepsilon^2_n \mu_n\big( (0,\varepsilon_n)^2 \cap T_n \big)}{\mu_n\big( T_n \setminus (0,\varepsilon_n)^2\big)} \xrightarrow[n \rightarrow \infty]{}\infty.
\end{equation}
By Lemma \ref{lem:polybounds} and the assumption $\varepsilon_n \asymp n^{-\kappa}$,
\begin{equation*}
\begin{split}
\mu_n\big( T_n \setminus (0,\varepsilon_n)^2\big) & \leqslant \sum_{i=1}^4 \mu_n\big( B^{(i)}_n\big) \\
& = \begin{cases} O(n^{-3+\kappa(2\alpha+1)}), & \kappa \in \big[1/(2\alpha +1),1\big),\\
o(n^{-2}), & \kappa \in \big(0, 1/(2\alpha +1)\big),
\end{cases}
\end{split}
\end{equation*}
and
\begin{equation*}
\varepsilon^2_n \mu_n\big( (0,\varepsilon_n)^2 \cap T_n \big) \gtrsim_{\alpha,\ell} n^{-2(1+\kappa-\alpha)}, \quad n \geqslant n_0.
\end{equation*}

Thus, when $\kappa \in \big(0, 1/(2\alpha +1)\big)$ we have
\begin{equation*}
M_n \gtrsim_{\alpha,\ell} \frac{n^{2(\alpha-\kappa)}}{o(1)},
\end{equation*}
whence \eqref{eq:muconv} holds if $\kappa \leqslant \alpha$.
In the case $\kappa \in \big[1/(2\alpha +1),1\big)$,
\begin{equation*}
M_n \gtrsim_{\alpha,\ell} \frac{n^{2(\alpha-\kappa-1)}}{O(n^{-3+\kappa(2\alpha+1)})}
\end{equation*}
and, consequently, \eqref{eq:muconv} holds provided that $\kappa < (2\alpha +1)/(2\alpha + 3)$.
It remains to note that for $\alpha \in (0,1/2)$,
\begin{equation}\label{eq:alphasmall}
\alpha < \frac{2\alpha+1}{2\alpha+3} < \frac{1}{2\alpha + 1},
\end{equation}
whereas for $\alpha \in [1/2,1)$,
\begin{equation}\label{eq:alphalarge}
\frac{1}{2\alpha + 1} \leqslant \frac{2\alpha+1}{2\alpha+3} \leqslant \alpha.
\end{equation}
The sufficiency of the asserted conditions can now be verified using the inequalities \eqref{eq:alphasmall} and \eqref{eq:alphalarge}.
\end{proof}

\subsection{Sketch of the proof of Proposition \ref{prop:trianglekernel}}\label{subsec:triangle}

\begin{figure}
  \myfigure{scale=0.9}{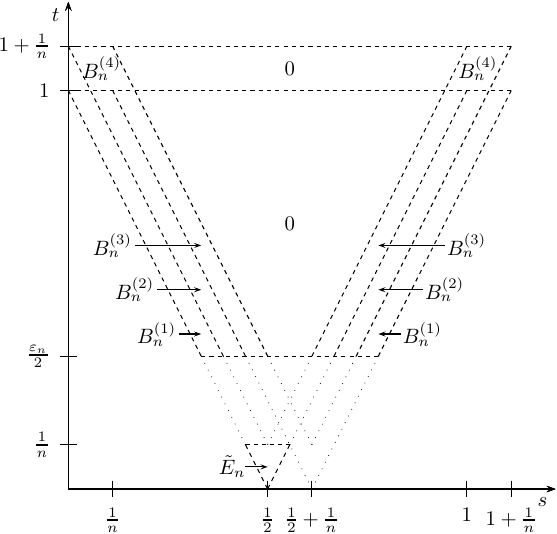}
  \caption{\label{fig:triangle} The redefined sets $\tilde{E}_n$, $B^{(1)}_n$, $B^{(2)}_n$, $B^{(3)}_n$, and $B^{(4)}_n$. The sets labeled with zeros are $\mu_n$-null sets for large $n$ under the specification of $g$ by \eqref{eq:gspec2}.} 
\end{figure}

Let now $g \in L^2(\R^2)$ be given by \eqref{eq:gspec2}.
We redefine the sets $\tilde{E}_n$, $B^{(1)}_n$, $B^{(2)}_n$, $B^{(3)}_n$, and $B^{(4)}_n$, $n \in \N$ as indicated in Figure \ref{fig:triangle}. Moreover, the measures $\mu_n$, $n\in\N$, are redefined accordingly.

\begin{proof}[Proof of Proposition \ref{prop:trianglekernel} (sketch)]
Similarly to the proof of Lemma \revised{\ref{lem:polybounds}}, we have
\begin{equation*}
\mu_n\big(\tilde{E}_n\big) = 
\int_{0}^{1/n} t f(t)^2 \ud t \gtrsim_{\ell} \int_{0}^{1/n} t^{-2\alpha+1} \ud t =
 \frac{n^{-2(1-\alpha)}}{2(1-\alpha)}
\end{equation*}
and
\begin{equation*}
\mu_n\big(B^{(2)}_n\big) = 
\frac{1}{n}\int^{1}_{\varepsilon_n/2} \big(f(t)-f(t-1/n)\big)^2 \ud t = O(n^{-3+\kappa(2\alpha+1)}). 
\end{equation*}Additionally, $\mu_n\big(B^{(4)}_n\big) = o(n^{-2})$. For the remaining two sets, we obtain
\begin{equation*}
\mu_n\big(B^{(1)}_n\big) =  
\frac{1}{n}\int^{1}_{\varepsilon_n/2} f(t)^2 \ud t \lesssim_{\ell} \frac{1}{n}\int^{1}_{\varepsilon_n/2} t^{-2\alpha} \ud t
\end{equation*}
and
\begin{equation*}
\mu_n\big(B^{(3)}_n\big) =  
\frac{1}{n}\int^{1}_{\varepsilon_n/2} f(t-1/n)^2 \ud t \lesssim_\ell \frac{1}{n}\int^{1}_{\varepsilon_n/2-1/n} t^{-2\alpha} \ud t,
\end{equation*}
and observing that $\varepsilon_n /2 \asymp \varepsilon_n /2 -1/n \asymp \varepsilon_n \asymp n^{-\kappa}$ (note that $\asymp$ is an equivalence relation), we have
\begin{equation}\label{eq:dominant}
\mu_n\big(B^{(1)}_n\big)+\mu_n\big(B^{(3)}_n\big)=O(n^{-1+\kappa(2\alpha-1)}).
\end{equation}

To prove the assertion, it suffices to show that
\begin{equation}\label{eq:ratio}
M_n \eqdefl \frac{\varepsilon_n^2 \mu_n\big(\tilde{E}_n\big)}{\sum_{i=1}^4 \mu_n\big(B^{(i)}_n\big)} \xrightarrow[n \rightarrow \infty]{}\infty.
\end{equation}
Since the contribution of \eqref{eq:dominant} is dominant in the denominator of \eqref{eq:ratio}, we have
\begin{equation*}
M_n \gtrsim_{\ell,\alpha} n^{(2\alpha-1)-\kappa(2\alpha+1)},
\end{equation*}
whence \eqref{eq:ratio} holds provided that $\kappa < (2\alpha-1)/(2\alpha+1)$.
\end{proof}

\section*{Acknowledgements}

I would like to thank Ole E.~Barndorff-Nielsen, Svend Erik Graversen, and Mark Podolskij for valuable discussions \revised{and two referees for helpful comments}.
Also, I acknowledge support 
from CREATES (DNRF78), funded by the Danish National Research Foundation, 
from the Aarhus University Research Foundation regarding the  project ``Stochastic and Econometric Analysis of Commodity Markets", and
from the Academy of Finland (project 258042).

\appendix

\section{Uniform convergence of functions of two variables} 

It is a well-known fact that if non-decreasing functions on $[0,1]$ converge pointwise to a continuous function, then the convergence is, in fact, uniform. 
In the proof of Theorem \ref{thm:powerlln} we invoke the following analogous result that applies to functions on $[0,1]^2$.

\begin{lem}[Uniform convergence]\label{lem:unifconv}
Let $f_1,f_2,\ldots$ be functions $[0,1]^2 \longrightarrow \R$ such that for any $n \in \N$,
\begin{equation*}
 f_n(s,t) \leqslant f_n(u,v) \quad \textrm{if $s \leqslant u$ and $t \leqslant v$,}
\end{equation*}
and let $f \in C([0,1]^2)$. If $f_n(s,t) \rightarrow f(s,t)$ for any $(s,t)\in[0,1]^2$, then $f_n \rightarrow f$ uniformly. 
\end{lem}

\begin{proof}
The assertion follows from a straightforward adaptation of the standard argument used in the univariate case (see, e.g., \cite[pp.~113--114]{Boas1996}). 
\end{proof}

\section{On the two-variable generalization of the c\`adl\`ag property}\label{app:skorohodspace}

We will review briefly the natural generalization of the \emph{c\`adl\`ag} property (continuity from the right with  finite limits from the left) for functions on $[0,1]^2$, following the formulation due to Neuhaus \cite{Neu1971}. 
To this end, we introduce for any $(s,t)\in [0,1]^2$ the quadrants
\begin{align*}
Q_1(s,t) &\eqdefl (s,1] \times (t,1],&
Q_2(s,t) &\eqdefl [0,s) \times (t,1],\\ 
Q_3(s,t) &\eqdefl [0,s) \times [0,t),&
Q_4(s,t) &\eqdefl (s,1] \times [0,t).
\end{align*}
The space $D([0,1]^2)$ consists of functions $f : [0,1]^2 \longrightarrow \R$ such that for any $(s,t) \in [0,1]^2$, the following two conditions hold.
\begin{itemize}
\item We have $f(s_n,t_n) \rightarrow f(s,t)$ if $(s_n,t_n)$ is a sequence in $Q_1(s,t)$ such that $(s_n,t_n) \rightarrow (s,t)$,
\item For any $i = 2,3,4$, there exists $\tilde{f}_i(s,t) \in \R$ that satisfies $f(s_n,t_n) \rightarrow \tilde{f}_i(s,t)$ if $(s_n,t_n)$ is a sequence in $Q_i(s,t)$ such that $(s_n,t_n) \rightarrow (s,t)$.
\end{itemize}
In other words, $f \in D([0,1]^2)$ is continuous from the direction of the first quadrant (cf.~the \emph{c\`ad} property) and has limits from the directions of the other three quadrants (cf.~the \emph{l\`ag} property). Clearly, we have $C([0,1]^2) \subset D([0,1]^2)$.

  The space $D([0,1]^2)$ can be endowed with the generalized Skorohod topology defined by Bickel and Wichura \cite{BW1971} and Neuhaus \cite{Neu1971}, which can be characterized in terms of convergence of sequences as follows. Let us denote by $\Lambda$ the class of mappings $\lambda : [0,1]^2 \longrightarrow [0,1]^2$ such that $\lambda(s,t) = \big(\lambda^{(1)}(s),\lambda^{(2)}(t)\big)$, where $\lambda^{(1)}$ and $\lambda^{(2)}$ are increasing bijections $[0,1]\longrightarrow[0,1]$.

\begin{defn}[Skorohod topology]
Let $f,f_1,f_2\ldots \in D([0,1]^2)$. We say that $f_n \rightarrow f$ in the \emph{Skorohod topology} if
 there exist $\lambda_1, \lambda_2,\ldots \in \Lambda$ such that 
\begin{equation}\label{eq:sko}
\sup_{(s,t)\in [0,1]^2} | f_n \circ \lambda_n (s,t) - f(s,t)|+  \sup_{(s,t)\in [0,1]^2} \| \lambda_n(s,t)-(s,t)\|  \xrightarrow[ n\rightarrow \infty]{} 0.
\end{equation}
\end{defn}

There is a \emph{Skorohod metric} on $D([0,1]^2)$ that is consistent with the convergence defined above (see \cite[p.\ 1289]{Neu1971}). Equipped with this metric, $D([0,1]^2)$ enjoys the usual properties of separability and completeness (i.e., it is a \emph{Polish} space), similarly to $D([0,1])$. 

We use the Skorohod topology merely as a technical tool to establish convergence in law in the proof of Theorem \ref{thm:powerclt}, using the relatively tractable tightness criteria for the Skorohod topology \cite[pp.~1665--1666]{BW1971}. Since the limit obtained in Theorem \ref{thm:powerclt} is a continuous random field, we may --- equivalently --- equip $D([0,1]^2)$ with the (non-separable) \emph{uniform topology}, thanks to the following result. 

\begin{lem}[Uniform convergence]\label{lem:skorohod}
Let $f_1,f_2,\ldots \in D([0,1]^2)$ and $f \in C([0,1]^2)$. Then, $f_n \rightarrow f$ in the Skorohod topology if and only if $f_n \rightarrow f$ uniformly.
\end{lem}

\begin{proof} It is obvious that uniform convergence implies convergence in the Skorohod topology. To show the converse, let us fix $\varepsilon>0$. Since $f$ is uniformly continuous, there exists $\delta>0$ such that $|f(s,t)-f(u,v)|<\varepsilon/2$ if  $\| (s,t) - (u,v) \| <\delta$.
Now, let $\lambda_1,\lambda_2,\ldots \in \Lambda$ be such that \eqref{eq:sko} holds. Then there exists $n_0 \in \N$ such that for all $n \geqslant n_0$,
\begin{equation*}
\sup_{(s,t) \in [0,1]^2} | f_n \circ \lambda_n (s,t) - f(s,t)|+  \sup_{(s,t)\in [0,1]^2} \| \lambda_n(s,t)-(s,t)\|  < \frac{\varepsilon}{2} \wedge \delta.
\end{equation*}
By the triangle inequality, we have thus for all $n \geqslant n_0$,
\begin{equation*}
\begin{split}
\sup_{(s,t)\in [0,1]^2}|f_n(s,t)-f(s,t)| & = \sup_{(s,t)\in [0,1]^2}|f_n\circ \lambda_n(s,t)-f \circ \lambda_n(s,t)| \\
& \leqslant \sup_{(s,t)\in [0,1]^2}|f_n \circ \lambda_n(s,t)-f(s,t)|\\ 
& \quad + \sup_{(s,t)\in [0,1]^2}|f(s,t)-f \circ \lambda_n(s,t)|  <\varepsilon,
\end{split}
\end{equation*}
which completes the proof.
\end{proof}

\section{Stable convergence lemma}

The following simple lemma is a key tool in proofs of stable convergence in law. It is certainly well-known and, indeed, used (implicitly) in several papers (e.g., \cite{BCP2011,BCPW2009}),
but due to lack of a reference, we provide a proof for the convenience of the reader.

\begin{lem}[Stable convergence]\label{lem:stable}
Let $\mathcal{U}$ and $\mathcal{V}$ be Polish spaces.
If $U,U_1,U_2,\ldots$ are random elements in $\mathcal{U}$ and $V$ is a random element in $\mathcal{V}$, all defined on a common probability space $(\Omega',\mathcal{F}',\prob')$, such that
\begin{equation*}
(U_n,V) \xrightarrow[n\rightarrow \infty]{L} (U,V),
\end{equation*}
then
\begin{equation*}
U_n \xrightarrow[n\rightarrow \infty]{L_{\sigma(V)}} U.
\end{equation*}
\end{lem}

\begin{proof}
We will use a monotone class argument. To this end, let $f \in  C(\mathcal{U},\R)$ be bounded and write
\begin{equation*}
\mathcal{M}_f \eqdefl \Big\{ X \in L^\infty(\Omega',\mathcal{F}',\prob') : \lim_{n \rightarrow \infty} \E'[f(U_n)X] = \E'[f(U)X] \Big\}.
\end{equation*}
Clearly, $\mathcal{M}_f$ is vector space that contains all constant random variables. \revised{Moreover, if $X$,\, $\tilde{X} \in L^\infty(\Omega',\mathcal{F}',\prob')$, }then
\begin{equation*}
|\E'[f(U_n)X] - \E'[f(U)X]|\lesssim_f |\E'[f(U_n)\tilde{X}] - \E'[f(U)\tilde{X}]| + \E'[|X-\tilde{X}|].
\end{equation*}
Hence, $\mathcal{M}_f$ is closed under uniform convergence and if $(\tilde{X}_n) \subset \mathcal{M}_f$ is such that $0 \leqslant \tilde{X}_1 \leqslant \tilde{X}_2 \leqslant \cdots \leqslant M$ for some constant $M>0$, then $\lim_{n \rightarrow \infty} \tilde{X}_n \in \mathcal{M}_f$. Now, note that
\begin{equation*}
\mathcal{C} \eqdefl \{ \varphi(V) : \textrm{$\varphi \in C(\mathcal{V},\R)$ is bounded} \}
\end{equation*}
is closed under multiplication and $\mathcal{C} \subset \mathcal{M}_f$ by the continuous mapping theorem. Thus, by the functional monotone class lemma \cite[p.\ 14]{DM1978}, $\mathcal{M}_f$ contains any bounded $\sigma(\mathcal{C})$-measurable random variable.
Since $\mathcal{V}$ is separable, we have $\sigma(V) = \sigma(\mathcal{C})$ and the assertion follows.
\end{proof}



\begin{thebibliography}{10}

\bibitem{AE1978}
D.~J. Aldous and G.~K. Eagleson (1978):
On mixing and stability of limit theorems.
{\em Ann. Probability} {\bf 6}(2), 325--331.

\bibitem{BNBV2010}
O.~E. Barndorff-Nielsen, F.~E. Benth, and A.~E.~D. Veraart \revised{(2014)}:
Modelling electricity forward markets by ambit fields.
\revised{\emph{Adv. Appl. Probab.} {\bf 46}(3), in press.}

\bibitem{BNBV2011}
O.~E. Barndorff-Nielsen, F.~E. Benth, and A.~E.~D. Veraart (2011):
Ambit processes and stochastic partial differential equations.
In {\em Advanced mathematical methods for finance}, pp.~35--74.
  Springer, Heidelberg.

\bibitem{BBV2012}
O.~E. Barndorff-Nielsen, F.~E. Benth, and A.~E.~D. Veraart (2012):
Recent advances in ambit stochastics with a view towards
  tempo-spatial stochastic volatility/intermittency.
arXiv:1210.1354

\bibitem{BCP2009}
O.~E. Barndorff-Nielsen, J.~M. Corcuera, and M.~Podolskij (2009):
Power variation for {G}aussian processes with stationary increments.
{\em Stochastic Process. Appl.} {\bf 119}(6), 1845--1865.

\bibitem{BCP2011}
O.~E. Barndorff-Nielsen, J.~M. Corcuera, and M.~Podolskij (2011):
Multipower variation for {B}rownian semistationary processes.
{\em Bernoulli} {\bf 17}(4), 1159--1194.

\bibitem{BCP2012}
O.~E. Barndorff-Nielsen, J.~M. Corcuera, and M.~Podolskij (2013):
Limit theorems for functionals of higher order differences of
  {B}rownian semi-stationary processes.
In A.~N. Shiryaev, S.~R.~S. Varadhan, and E.~L. Presman, eds.,
  {\em Prokhorov and Contemporary Probability Theory}, pp.~69--96. Springer, Berlin.

\bibitem{BCPW2009}
O.~E. Barndorff-Nielsen, J.~M. Corcuera, M.~Podolskij, and J.~H.~C. Woerner (2009):
Bipower variation for {G}aussian processes with stationary
  increments.
{\em J. Appl. Probab.} {\bf 46}(1), 132--150.



\bibitem{BNG2011}
O.~E. Barndorff-Nielsen and S.~E. Graversen (2011):
Volatility determination in an ambit process setting.
{\em J. Appl. Probab.} {\bf 48A}, 263--275.

\bibitem{BNG2013}
O.~E. Barndorff-Nielsen, M.~S.~Pakkanen, and J.~Schmiegel (2013):
Assessing relative volatility/intermit\-tency/energy dissipation.
arXiv:1304.6683

\bibitem{BS2005}
O.~E. Barndorff-Nielsen and J.~Schmiegel (2005):
Spatio-temporal modeling based on {L}\'evy processes, and its
  applications to turbulence.
{\em Russian Math. Surveys} {\bf 59}(1), 65--90.

\bibitem{BS2006}
O.~E. Barndorff-Nielsen and J.~Schmiegel (2007):
Ambit processes:~with applications to turbulence and tumour growth.
In {\em Stochastic analysis and applications}, vol.~2 of {\em Abel
  Symp.}, pp.~93--124. Springer, Berlin.

\revised{
\bibitem{B2010}
A.~Basse-O'Connor (2010):
Representation of Gaussian semimartingales with applications to the covariance function.
{\em Stochastics} {\bf 82}(4), 381--401.}


\bibitem{Ber1967}
S.~M. Berman (1967):
A version of the {L}\'evy-{B}axter theorem for the increments of
  {B}rownian motion of several parameters.
{\em Proc. Amer. Math. Soc.} {\bf 18}, 1051--1055.

\bibitem{BW1971}
P.~J. Bickel and M.~J. Wichura (1971):
Convergence criteria for multiparameter stochastic processes and some
  applications.
{\em Ann. Math. Statist.} {\bf 42}, 1656--1670.

\bibitem{Bil1999}
P.~Billingsley (1999):
{\em Convergence of probability measures}, 2nd ed.
Wiley, New York.

\bibitem{Boas1996}
R.~P. Boas (1996):
{\em A primer of real functions}, 4th ed.
Mathematical Association of America, Washington D.C.

\bibitem{Bre2011}
J.-C. Breton (2011):
On the rate of convergence in non-central asymptotics of the
  {H}ermite variations of fractional {B}rownian sheet.
{\em Probab. Math. Statist.} {\bf 31}(2), 301--311.

\bibitem{CW2000}
G.~Chan and A.~T.~A.~Wood (2000):
Increment-based estimators of fractal dimension for two-dimensional surface data.
{\em Statist. Sinica} {\bf 10}(2), 343--376.

\revised{
\bibitem{C2003}
P.~Cheridito (2004):
Gaussian moving averages, semimartingales and option pricing. 
{\em Stochastic Process. Appl.} {\bf 109}(1), 47--68. } 

\bibitem{CHPP2012}
J.~M. Corcuera, E.~Hedevang, M.~S. Pakkanen, and M.~Podolskij (2013):
Asymptotic theory for {B}rownian semi-stationary processes with
  application to turbulence.
  {\em Stochastic Process. Appl.} {\bf 123}(7), 2552--2574.
  

\bibitem{DM1978}
C.~Dellacherie and P.-A.~Meyer (1978):
{\em Probabilities and potential}.
Hermann, Paris.


\bibitem{Deo1989}
C.~M. Deo (1989):
A functional central limit theorem for the quadratic variation of a
  class of {G}aussian random fields.
{\em Canad. J. Statist.} {\bf 17}(2), 247--251.

\bibitem{Guy1987}
X.~Guyon (1987):
Variations de champs gaussiens stationnaires:\ application a l'identification.
{\em Probab. Theory Related Fields} {\bf 75}(2), 179--193. 

\revised{\bibitem{HL1999}
F.~Hirsch and G.~Lacombe (1999):
{\em Elements of functional analysis}.
Springer, New York.}

\bibitem{JP2012}
J.~Jacod and P.~Protter (2012):
{\em Discretization of processes}.
Springer, Heidelberg.

\bibitem{Kal2002}
O.~Kallenberg (2002):
{\em Foundations of modern probability}, 2nd ed.
Springer, New York.

\revised{\bibitem{K1950}
K.~Karhunen (1950):
\"{U}ber die Struktur station\"arer zuf\"alliger Funktionen.
{\em Ark. Mat.} {\bf 1}(2), 141--160.}

\bibitem{Kaw1975}
T.~Kawada (1975):
The {L}\'evy-{B}axter theorem for {G}aussian random fields: a sufficient condition.
{\em Proc. Amer. Math. Soc.} {\bf 53}(2), 463--469. 

\bibitem{LR2001}
G.~Lang and F.~Roueff (2001): 
Semi-parametric estimation of the {H}\"older exponent of a stationary
  {G}aussian process with minimax rates.
{\em Stat. Inference Stoch. Process.} {\bf 4}(3), 283--306.

\bibitem{LL2001}
E.~H. Lieb and M.~Loss (2001):
{\em Analysis}, 2nd ed.
American Mathematical Society, Providence.

\bibitem{Neu1971}
G.~Neuhaus (1971):
On weak convergence of stochastic processes with multidimensional
  time parameter.
{\em Ann. Math. Statist.} {\bf 42}, 1285--1295.

\bibitem{Nua1995}
D.~Nualart (1995):
{\em The {M}alliavin calculus and related topics}.
Springer, New York.

\bibitem{NP2005}
D.~Nualart and G.~Peccati (2005):
Central limit theorems for sequences of multiple stochastic
  integrals.
{\em Ann. Probab.} {\bf 33}(1), 177--193.

\bibitem{PT2005}
G.~Peccati and C.~A. Tudor (2005):
Gaussian limits for vector-valued multiple stochastic integrals.
In {\em S\'eminaire de {P}robabilit\'es {XXXVIII}}, vol.~1857 of
  {\em Lecture Notes in Math.}, pp.~247--262. Springer, Berlin.

\bibitem{Ren1963}
A.~R{\'e}nyi (1963):
On stable sequences of events.
{\em Sankhy\=a Ser.~A} {\bf 25}, 293--302.

\bibitem{Rev2009b}
A.~R{\'e}veillac (2009):
Convergence of finite-dimensional laws of the
weighted quadratic variations process for some
fractional {B}rownian sheets,
{\em Stoch. Anal. Appl.} {\bf 27}(1), 51--73.

\bibitem{Rev2009}
A.~R{\'e}veillac (2009):
Estimation of quadratic variation for two-parameter diffusions.
{\em Stochastic Process. Appl.} {\bf 119}(5), 1652--1672.

\bibitem{RST2012}
A.~R{\'e}veillac, M.~Stauch, and C.~A. Tudor (2012):
Hermite variations of the fractional {B}rownian sheet.
{\em Stoch. Dyn.} {\bf 12}(3), 1150021, 21 pp.


\bibitem{BNES2005}
J.~Schmiegel, O.~E. Barndorff-Nielsen, and H.~C.~Eggers (2005):
A class of spatio-temporal and causal stochastic processes with application to multiscaling and multifractality.
{\em S. Afr. J. Sci.} {\bf 101}, 513--519. 

\bibitem{SCEPG2004}
J.~Schmiegel, J.~Cleve, H.~C.~Eggers, B.~R.~Pearson, and M.~Greiner (2004): Stochastic energy-cascade model for $(1+1)$-dimensional fully developed turbulence.
{\em Phys. Lett. A} {\bf 320}(4), 247--253. 



\bibitem{Sou2001}
P.~Soulier (2001):
Moment bounds and central limit theorem for functions of {G}aussian
  vectors.
{\em Statist. Probab. Lett.}  {\bf 54}(2), 193--203.

\bibitem{Taqqu1977}
M.~S. Taqqu (1977):
Law of the iterated logarithm for sums of non-linear functions of
  {G}aussian variables that exhibit a long range dependence.
{\em Z. Wahrscheinlichkeitstheorie und Verw. Gebiete}
  {\bf 40}(3), 203--238.

\end{thebibliography}
\end{document}